\documentclass{amsart}
\usepackage{amsfonts}
\usepackage{amsthm} 
\usepackage{amsmath}
\usepackage{amssymb}
\usepackage{amscd}
\usepackage[latin2]{inputenc}
\usepackage{t1enc}
\usepackage{mathrsfs}
\usepackage{indentfirst}
\usepackage{graphicx}
\usepackage{graphics}
\usepackage{hyperref}
\usepackage{pict2e}
\usepackage{tikz-cd}
\numberwithin{equation}{section}
\usepackage{epstopdf}

\newenvironment{spmatrix}
  {\left(\begin{smallmatrix}}
  {\end{smallmatrix}\right)}
\newtheorem{thm}{Theorem}[section]
\newtheorem*{thmA}{Theorem A}
\newtheorem*{thmB}{Theorem B}
\newtheorem*{thmC}{Theorem C}

\newtheorem{lem}[thm]{Lemma}
\newtheorem{prop}[thm]{Proposition}

\newtheorem{conj}[thm]{Conjecture}
\newtheorem*{conjA}{Conjecture A}

\newtheorem{rem}[thm]{Remark}
\theoremstyle{definition}

\newtheorem{defn}[thm]{Definition}
\theoremstyle{remark}

\def\A{\mathbb{A}}
\def\C{\mathbb{C}}

\def\P{\mathbb{P}}
\def\Q{\mathbb{Q}}
\def\R{\mathbb{R}} 
\def\Z{\mathbb{Z}}
\def\F{\mathbb{F}}
\def\begcd{\begin{tikzcd}}
\def\endcd{\end{tikzcd}}
\def\begenum{\begin{enumerate}}
\def\endenum{\end{enumerate}}
\def\begpmat{\begin{pmatrix}}
\def\endpmat{\end{pmatrix}}
\def\a{\alpha}

\def\e{\varepsilon}
\def\p{\prime}
\def\g{\gamma}
\def\s{\sigma}

\def\k{\kappa}

\def\O{\mathcal{O}}
\def\fa{\mathfrak{a}}

\def\fd{\mathfrak{d}}

\def\fl{\mathfrak{l}}

\def\fp{\mathfrak{p}}
\def\fq{\mathfrak{q}}

\def\fM{\mathfrak{M}}
\def\fN{\mathfrak{N}}
\def\fQ{\mathfrak{Q}}

\newcommand{\ovl}{\overline}
\newcommand{\Tr}{\operatorname{Tr}}
\newcommand{\id}{\operatorname{id}}
\newcommand{\co}{\operatorname{c}}
\newcommand{\SL}{\operatorname{SL}}
\newcommand{\GL}{\operatorname{GL}}
\newcommand{\SU}{\operatorname{SU}}

\begin{document}

\title[Bianchi modular symbols and $p$-adic $L$-functions]{Bianchi modular symbols and $p$-adic $L$-functions}
\author[Jaesung Kwon]{JAESUNG KWON}
\address{Ulsan National Institute of Science and Technology \\ Ulsan \\ Korea}\email{jaesung.math@gmail.com}


\begin{abstract} In the present paper, we construct the integral special $L$-value and integral valued $p$-adic $L$-function of $p$-ordinary Bianchi Hecke eigenforms.
We obtain the functional equation of the $L$-function of Bianchi modular forms. By using the functional equation, we compute the approximate functional equation, which is a fast convergent series of $L$-values. 
By studying the fast convergent series of special $L$-values, we prove that the first homology group of the Bianchi modular manifold is generated by the special Bianchi modular symbols. As a corollary of the generation result, we have that the $\mu$-invariant of some isotopic components of the $p$-adic $L$-function of certain Bianchi Hecke eigenforms vanishes for a positive proportion of ordinary prime ideals. By studying a $\ell$-adic representation attached to a Bianchi Hecke eigenform, we show that the density of the set of ordinary prime ideals is one. Also we obtain the result of residual non-vanishing of the integral $L$-values.
\end{abstract} 

\maketitle
\setcounter{tocdepth}{1} 
\tableofcontents

\section{Introduction}
The vanishing of the $\mu$-invariant of the $p$-adic $L$-function of modular form is a longstanding conjecture in Iwasawa theory, which was introduced by Greenberg \cite{greenberg1999iwasawa}. Kim-Sun \cite{kim2017modular} gave a proof of the result toward the Greenberg's conjecture for elliptic modular forms by studying the distribution of special modular symbols in the first homology group of the modular curve.

In the present paper, we define an integral valued $p$-adic $L$-function of Bianchi modular form, which interpolates the integral $L$-values. Also we show that the $\mu$-invariant of some isotopic components of the $p$-adic $L$-function of Bianchi modular form is vanishing, which is a generalization of Kim-Sun \cite{kim2017modular}.

\subsection{Generation of homology groups by special modular symbols}

Let $F$ be an imaginary quadratic field, $h_F$ the class number of $F$, $\{\fa_i\}_{i=1}^{h_F}$ a representative set of the class group of $F$ such that $\fa_i^{-1}$ is integral for each $1\leq i\leq h_F$, $a_i\in\A_F^\times$ an idelic representative of $\fa_i$, and $\fN$ an integral ideal of $F$ such that $[\Z:\fN\cap\Z]>3$. Let $\mathscr{H}_3$ be the hyperbolic upper half space of real dimension three, $U_1(\fN)$ the subgroup of $\GL_2(\widehat{\O}_F)$ defined in the subsection \ref{subsection:notation}, and $\ovl{\Gamma_1^i(\fN)}$ the congruence subgroup of $\SL_2(\O_F)$ defined in the section \ref{section:conj:on:mod:symb}.
Let us define the Bianchi modular manifold $Y_1(\fN)$ of level $U_1(\fN)$ by
$$
Y_1(\fN):=\GL_2(F)\backslash\GL_2(\A_F)/\C^\times\operatorname{U}_2(\C) U_1(\fN)\cong\coprod_{i=1}^{h_F}\ovl{\Gamma_1^i(\fN)}\backslash\mathscr{H}_3=:\coprod_{i}^{h_F} Y_1^i(\fN).
$$
Let us denote by $X_1(\fN)$ the Satake compatification of $Y_1(\fN)$. For $\a,\beta\in\P^1(F)$, let us denote by $\{\a,\beta\}_{\ovl{\Gamma_1^i(\fN)}}$ the image of the geodesic from $\a$ to $\beta$ in $X_1^i(\fN)$. Then we can consider $\{\a,\beta\}_{\ovl{\Gamma_1^i(\fN)}}$ to be an element of $H_1(X_1^i(\fN),\ovl{\Gamma_1^i(\fN)}\backslash\mathbb{P}^1(F),\Z)$. Then the classes $\{\a,\beta\}_{\ovl{\Gamma_1^i(\fN)}}$ are called modular symbols of level $\ovl{\Gamma_1^i(\fN)}$ and weight $(2,2)$.

Let $\fl$ be a prime ideal of $F$ with sufficiently large norm $N(\fl)$, $\O_{F,\fl}$ the $\fl$-adic completion of the integer ring $\O_F$ of $F$, and $\varpi_\fl$ a uniformizer of $\O_{F,\fl}$.
In Section \ref{section:conj:on:mod:symb}, we construct the modular symbol $\xi_{n,\fl}^{(i)}(a)=\{\infty,r_{n,\fl}(a)\}_{\ovl{\Gamma_1^i(\fN)}}$ in $X_1^i(\fN)$ corresponds to the finite adele $a\varpi_\fl^{-n}\in\A_{F}$. 
We consider the submodule $M_{n,\fl}$ of $H_1(X_1(\fN),\Z)$, which is defined by
$$
M_{n,\fl}:=\bigoplus_{i=1}^{h_F}\langle\xi_{n,\fl}^{(i)}(a):a\in \O_{F,\fl}^\times\rangle\cap H_1(X_1(\fN),\Z).
$$
We show that $M_{n,\fl}$ is of full-rank in $H_1(X_1(\fN),\Z)$ for almost all prime ideals $\fl$:

\begin{thmA}[Theorem \ref{vertical:direction} and \ref{full:rank:horizontal}] Let $\fl$ be a prime ideal and $n\geq 22$ an integer. Then we have the following:
\item[(1)]
$M_{\fl,n}$ is of full rank in $H_1(X_1(\fN),\Z)$ for sufficiently large $N(\fl)$.
\item[(2)]
There is an arithmetic progression $\mathfrak{X}_n$ of ideals such that $M_{n,\fp}\otimes\mathbb{F}_p=H_1(X_1(\fN),\ovl{\mathbb{F}}_q)$ for prime ideals $\fp\in\mathfrak{X}_n$ and sufficiently large prime numbers $q$.
\end{thmA}

\subsection{Integral $L$-values and residual non-vanishing}
Let $p$ be an odd prime number coprime to $h_F D_F|\O_F^\times|\fN$ such that $F$ does not contain any primitive $p$-th power roots.
Let $f$ be a normalized Bianchi Hecke eigenform of weight $(2,2)$ and level $U_1(\fN)$.
In Section \ref{conj:eisenstein:special}, we define the compact support cohomology class $\eta_{f,\operatorname{c}}\in H^1_{\operatorname{c}}(Y_1(\fN),\ovl{\Z}_p)$ attached to $f$. Let us denote by $\eta_{f,\operatorname{c},i}$ the $i$-th part of $\eta_{f,\operatorname{c}}$ with respect to the map $X_1^i(\fN)\hookrightarrow X_1(\fN)$.

For a finite order Hecke character $\phi:F^\times\backslash\A_F^\times\rightarrow\C^\times$ of prime power conductor $\fl^n$, let us define the twisted modular symbol by 
$$
\Lambda^{(i)}(\phi):=\sum_{r\text{ mod }\fl^n}\phi(a_i)\phi_\fl(r)\xi_{n,\fl}^{(i)}(r)\in H_1(X_1^i(\fN),\ovl{\Gamma_1^i(\fN)}\backslash\mathbb{P}^1(F),\Z[\phi]).
$$
Then we define the integral $L$-value by
$$
\mathcal{L}_f(\phi):=\sum_{i=1}^{h_F}\Lambda^{(i)}(\phi)\cap\eta_{f,\operatorname{c},i}\in\ovl{\Z}_p,
$$
where $\cap:H_1(X_1(\fN),\bigcup_i\ovl{\Gamma_1^i(\fN)}\backslash\mathbb{P}^1(F),\ovl{\Z}_p)\times H^1_{\rm{c}}(Y_1(\fN),\ovl{\Z}_p)\rightarrow\ovl{\Z}_p$ is the Lefschetz-Poincar\'{e} duality pairing. Also we define the algebraic $L$-value by
$$
L_f(\phi):=\frac{|D_F||\O_F^\times| G(\phi)L(1,f\otimes\phi)}{8\pi^2\Omega_{f}}\in\ovl{\Q}.
$$ 
Let us define the non-Eisenstein condition introduced in Namikawa \cite{Namikawa}:
$$
(\mathbf{Non\text{-}Eis})\ a_f(q\O_F)-N(q)-1\in\ovl{\Z}_p^\times\text{ for some prime element }q\equiv 1\ (\text{mod }p\fN).
$$
Note that this condition is related to the existence of the irreducible residual Galois representation attached to $f$ mod $p$ (see Scholze \cite{scholze2015torsion}). Also note that in Section \ref{p:adic:l:function:mu:invariants}, if (\textbf{Non-Eis}) holds, then we can see that the ratio of the integral $L$-value $\mathcal{L}_f(\phi)$ and the algebraic $L$-value $L_f(\phi)$ is a $p$-adic unit. 

Under the assumption (\textbf{Non-Eis}), Namikawa \cite{Namikawa} shows that $\mathcal{L}_f(\phi)$ is a $p$-adic unit for infinitely many Hecke characters when $h_F=1$. By using this result, we obtain the following result toward the generalization of Stevens' result \cite{stevens1958cuspidal}:
\begin{thmB}[Theorem\ \ref{homology:generate:padic:unit}]
Assume $(\mathbf{Non\text{-}Eis})$, then $\mathcal{L}_f(\phi)$ is a $p$-adic unit for some finite order Hecke characters $\phi$.
\end{thmB}

\subsection{Vanishing of $\mu$-invariant} 
To discuss Greenberg's conjecture, we need to define $p$-adic $L$-function for Bianchi modular forms and its $\mu$-invariant. 
Namikawa \cite{Namikawa2} constructed a two-variable $p$-adic $L$-function for nearly $p$-ordinary cohomological automorphic forms of $\GL_2$ over general number fields, and proved the non-vanishing of $L$-values for certain weights. Williams \cite{williams2017padic} defined a $p$-adic $L$-function of Bianchi modular form by using the overconvergent modular symbols, and proved the control theorem on the space of modular symbols. 

Let $p$ be an odd prime coprime to $h_F|\O_F^\times|\fN\fa_i^{-1}$ for each $i$, and $\fp$ a prime ideal of $F$ lying over $p$. Let $\phi$ be a finite Hecke character of $\fp$-power conductor. 
Let $f$ be a normalized $\fp$-ordinary Bianchi Hecke eigenform of weight $(2,2)$ and level $U_1(\fN)$. Let $\phi$ be a finite order Hecke character over $F$ of prime power conductor $\fp^n$.
In Section \ref{p:adic:l:function:mu:invariants}, we construct an integral $p$-adic $L$-function $L_\fp(s,f,\phi)$ by taking the $p$-adic Mazur-Mellin transformation on the $p$-adic measure $\nu_f$, 
which is defined by the modular symbols $\xi^{(i)}_{n,\fp}$ and the cohomology classes $\eta_{f,c,i}$. 
Then our $p$-adic $L$-function $L_\fp(s,f,\phi)$ interpolates the integral $L$-value $\mathcal{L}_f(\phi)$ (see Proposition \ref{integral:lvalue:interpolation}). 
We suggest the Bianchi modular version of Greenberg's conjecture:
\begin{conjA}[Conjecture \ref{greenberg:conj:bianchi:version}] Assume $(\mathbf{Non\text{-}Eis})$, then $\mu\big(L_\fp(s,f,\phi)\big)=0$.
\end{conjA}
Let $\omega:(\O_F/\fp)^\times\rightarrow\O_{F,\fp}^\times$ be the Teichm\"{u}ller character. By using Theorem A, we obtain the result toward Conjecture A:

\begin{thmC}[Theorem\ \ref{main:result}] Let $f$ be a normalized Bianchi modular Hecke eigenform of parallel weight $(2,2)$ with trivial central character and rational Hecke eigenvalues. Assume that $f$ is neither CM nor the lifting of a classical elliptic modular newform.  Suppose $(\mathbf{Non\text{-}Eis})$, then for a positive proportion of ordinary prime ideals $\fp$, we have $\mu\big(L_\fp(s,f\otimes\psi,\phi\omega^j )\big)=0$ for some Hilbert characters $\psi$ of $F$ and integers $0\leq j<p-1$.
\end{thmC}

\subsection{Skectch of the proof of Theorem A} 
Let $\fl$ be a prime ideal of $F$ whose norm $\ell$ is coprime to $\fa_i^{-1}\fN$ for each $i$.
For an element $b\in\O_{F,\fl}$ and an open set $U=a_0+\fl^v\O_{F,\fl}$, let us introduce the special modular symbol:
$$
\Upsilon_{n,\fl}^{(i)}(U,b):=\frac{1}{\ell^{n-v}}\sum_{a\equiv a_0 (\fl^v)}\mathbf{e}_F\Big(\frac{-ab}{\varpi_\fl^n }\Big)\cdot\xi^{(i)}_{n,\fl}(a).
$$
To prove Theorem A, we have to estimate the integral of a Bianchi modular form along the special modular symbol. 
By using the arguments of Kim-Sun \cite{kim2017modular}, Kwon-Sun \cite{kwon2020nonvanishing} and Luo-Ramakrishnan \cite{LuoRama}, the estimation of the integral is given as follows:

\begin{equation}\label{special:cycle:modular:form:integral:estimation}
\langle\Upsilon_{n,\fl}^{(i)}(U,\beta),f\rangle=\frac{a_{f}(\beta\fa_i)}{N(\beta\fa_i)}+o(1)\text{ for }\beta\in\fa_i^{-1}\text{ and }n,\ell\gg 0,
\end{equation}
where $\langle\cdot,\cdot\rangle$ is the pairing induced by the Lefschetz-Poincar\'{e} duality pairing $\cap$. Note that as $\langle\cdot,\cdot\rangle$ is non-degenerate, we can change the full-rankness problem into the non-vanishing one. By using the equation (\ref{special:cycle:modular:form:integral:estimation}), we verify the non-vanishing of the pairing. From this, we deduce Theorem A.

\subsection{Lattice estimation}
To obtain the equation (\ref{special:cycle:modular:form:integral:estimation}), we have to compute the estimation on the number of points in certain lattice in bounded region. To do this, we have to study the properties of the norm map $N_{F/\Q}:F^\times\rightarrow\Q^\times$. If $F$ is imaginary quadratic, then the norm map is given by the square of the absolute value, which is easy. But the difficulty arises for the case of general number fields. For example, if $F=\Q(\sqrt{D})$ is real quadratic, then the norm form $(a,b)\mapsto N(a+b\sqrt{D})a=a^2-b^2D$ is not positive definite.

\subsection{$\ell$-adic Galois representations and ordinary primes} 
In Theorem A, we consider the ordinary primes in some arithmetic progression. Therefore, we have to prove that the density of the ordinary primes is positive.
If a normalized Bianchi Hecke eigenform $f$ of weight $(2,2)$ with trivial central character is neither CM nor the lifting of a classical elliptic modular form, then there is a $\ell$-adic Galois representation $\rho_{f,\ell}$ such that the image of $\rho_{f,\ell}$ is dense in $\GL_2(\Q_\ell)$ and $\text{Tr}\big(\rho_{f,\ell}(\text{Frob}_{v})\big)$ is the Hecke eigenvalue $a_f(v)$ of $T_v$ for a set of places $v$ of $F$ of density one (see Taylor \cite{taylor1994ladic}). If the coefficient field of $f$ is rational, then we can prove the existence of a set of ordinary primes of density one by using the argument by Serre \cite{serre1981quelques}.

\subsection{Technical Difficulties} 
To deal with the general class number, we use the adelic language. In Kwon-Sun \cite{kwon2020nonvanishing}, we can see the similar techniques.


Consider the module
$$
M_{n,\fl,0}:=\bigg\langle\sum_{i=1}^{h_F}\xi_{n,\fl}^{(i)}(a):a\in \O_{F,\fl}^\times\bigg\rangle\cap H_1(X_1(\fN),\Z),
$$
which is a submodule of $M_{n,\fl}$. If we prove that $M_{n,\fl,0}$ is of full-rank in $H_1(X_1(\fN),\Z)$, then we can remove the twisting $\otimes\psi$ in the statement of Theorem C. But the non-trivial class group is an obstruction for proving the full-rankness of $M_{n,\fl,0}$. Thus we consider a bigger module $M_{n,\fl}$ and the twisting $\otimes\psi$.

\subsection{Notations}\label{subsection:notation}
Let us provide some notations which is used globally in this paper.
\begin{itemize}
\item
Let $F$ be an imaginary quadratic field, $\O_F$ the integer ring of $F$, $\fd_F$ the different ideal of $F$, $d_F$ a finite idele of $\fd_F$, $D_F$ the discriminant of $F$, and $N$ the norm map of $F/\Q$.
\item
For a $\Z$-algebra $A$, let us set $\widehat{A}:=A\otimes_\Z\widehat{\Z}$.
\item
Let $\A_F$ be the adele ring of $F$, $\A_F^{(\infty)}$ the ring of finite adeles of $F$, and $|\cdot|_{\A_F}=|\cdot|_{\infty}\times|\cdot|_{\A_F^{(\infty)}}$ the idelic norm.
\item
For a place $v$ of $F$, let us denote by $\O_{F,v}$ the $v$-adic completion of $\O_F$. 
\item
Let us choose a uniformizer $\varpi_v$ of $\O_{F,v}$. Let us embed $\varpi_v$ into $\A_F^\times$ by $\varpi_v\mapsto(1,\cdots,1,\varpi_v,1,\cdots,1)$. For an ideal $\fa$ of $F$, define $\varpi_\fa:=\prod_{v|\fa}\varpi_v^{\text{ord}_v(\fa)}$.
\item
Let $\fN$ be a non-zero integral ideal of $F$ such that $[\Z:\fN\cap\Z]>3$. Define subgroups $U_0(\fN)$ and $U_1(\fN)$ of $\GL_2(\widehat{\O}_F)$ by
\begin{align*}
U_0(\fN)&:=\bigg\{\begin{pmatrix} a & b \\ c & d \end{pmatrix}\in \GL_2(\widehat{\O}_F) : c\in \widehat{\fN} \bigg\} , \\
U_1(\fN)&:=\bigg\{\begin{pmatrix} a & b \\ c & d \end{pmatrix}\in U_0(\fN) : d\in 1+\widehat{\fN} \bigg\} .
\end{align*}
\item
Let Cl$(F)$ be the class group of $F$, $h_F$ the class number of $F$, and $H_F$ the group of Hilbert characters of $F$. Note that Cl$(F)\cong F^\times\backslash\A_F^\times/\C^\times\widehat{\O}_F^\times$.
\item
Without loss of generality, we can choose a representative set $\{\fa_i\}_{i=1}^{h_F}$ of $\text{Cl}(F)$ such that $\fa_1=\O_F$, $\fa_i^{-1}$ is integral and coprime to $\fN$ for each $i$. Let us denote $a_i:=\varpi_{\fa_i}$.
\item
Let $\mathbf{e}_F:\A_F/F\rightarrow\C^\times$ be an additive character defined by
$$
\mathbf{e}_F(x_\infty x^{(\infty)}):=\mathbf{e}\big(\Tr_{\C/\R}(x_\infty)\big)\cdot\prod_p\prod_{\fp|p}\mathbf{e}_p\big(\Tr_{F_\fp/\Q_p}(x^{(\infty)})\big),
$$
where $\mathbf{e}(z):=\operatorname{exp}\big({2\pi\sqrt{-1}z}\big)$ and $\mathbf{e}_p\big(\sum_{j}c_jp^j\big):=\mathbf{e}\big(\sum_{j<0}c_jp^j\big).$
\item
We fix an embedding $\ovl{\Q}\hookrightarrow \ovl{\Q}_\ell$ for each rational primes $\ell$.
\end{itemize}

\section{Cusp forms on GL(2)}\label{cuspform}
In this section, we give the definition of \emph{cuspidal automorphic form} on $\GL(2)$ over $F$, which is also called \emph{Bianchi cusp form}. Note that by Hida \cite[Corollary 2.2]{hida1994critical}, all the cusp forms of non-parallel weights are trivial, thus we only consider the parallel weight.

Let $I_F:=\{\id,\co\}$ be the set of embeddings of $F$, and $\Z[I_F]$ the free $\Z$-module generated by $I_F$. Let us denote $(k_{\id},k_{\co}):=\sum_{\s\in I_F}k_\s\cdot\s\in\Z[I_F]$ for $k\geq 2$.
Let us denote $\mathbf{x}:=\begin{spmatrix} X \\ Y \end{spmatrix}$, where $X$ and $Y$ are indeterminates. 
Let us denote by $L_d(R)$ the space of homogeneous polynomials in two variables $X,Y$ of degree $d\geq 0$ with coefficients $R$. Note that there is a usual action of GL(2) on $L_d(R)$.

\begin{defn}[{Hida \cite[Section 3]{hida1994critical}}]\label{adel:cuspform:defn} Let $\chi:F^\times\backslash\A_F^\times\rightarrow\C^\times$ be a Hecke character of modulus $\fN$ such that $\chi(z_\infty)=|z_\infty|^{2(2-k)}$ for $z_\infty\in \C^\times$. A \emph{(Bianchi) cusp form} on $\GL_2(\A_F)$ of weight $(k,k)$ and level $\fN$ with a central character $\chi$ is a $C^\infty$-function $f:\GL_2(\A_F)\rightarrow L_{2k-2}(\C)$ such that
\begenum
\item[(1)] $D_\sigma f=\big(\frac{k^2-2k}{2}\big)f$ for each $\s\in I_F$, where $D_\sigma$ is the Casimir operator for $\SL(2)$ corresponding to $\s$.
\item[(2)] $f(\g z g h)(\mathbf{x})=\chi(z)\chi_\fN(h^{(\infty)})f(g)\big(h_\infty\mathbf{x}\big)$ for $\g\in \GL_2(F)$, $z\in\A_F^\times$, $g\in\GL_2(\A_F)$ and $h=h_\infty h^{(\infty)}\in\SU_2(\C)U_0(\fN)$, where $\chi_\fN\big(\begin{spmatrix} a & b \\ c & d \end{spmatrix}\big):=\prod_{v|\fN}\chi(d_v)$.
\item[(3)] $\int_{F\backslash \A_F}f\big(\begin{spmatrix} 1 & x \\ 0 & 1 \end{spmatrix}g\big)(\textbf{x})dx=0$ for $g\in \GL_2(\A_F)$, where $dx$ is a Haar measure on $F\backslash\A_F$.
\endenum
From now on, we denote by $S_{k}(\fN,\chi)$ the space of the aforementioned cusp forms.
\end{defn}

Let us denote by $K_j$ the $j$-th modified Bessel function of second kind. Let $W_{k}:\C^\times\rightarrow L_{2k-2}(\C)$ be the \emph{Whittaker function}, which is defined by 
\begin{equation*}
W_{k}(y_\infty)(\textbf{x}):=\sum_{m=0}^{2k-2}\begin{spmatrix} 2k-2 \\ m \end{spmatrix} \bigg(\frac{y_\infty}{\sqrt{-1}|y_\infty|}\bigg)^{k-1-m} K_{m+1-k}(4\pi|y_\infty|)X^{2k-2-m}Y^{m}.
\end{equation*}
Then we have the \emph{Fourier-Whittaker expansion} of cusp forms:

\begin{prop}[Hida {\cite[Theorem 6.1]{hida1994critical}}]\label{four:whit:exp:hida}
Let $\mathscr{F}$ be the set of fractional ideals of $F$. For each $f\in S_{k}(\fN,\chi)$, there exists a function $a_f:\mathscr{F}\rightarrow\C$ satisfying the following properties:
\begenum
\item[(1)] $a_f(\fa)=0$ unless $\fa\in\mathscr{F}$ is integral. 
\item[(2)] For $x,y\in\A_F$, we have the Fourier-Whittaker expansion of $f$ as follows:
$$
f\bigg(\begin{pmatrix} y & x \\ 0 & 1 \end{pmatrix}\bigg)(\mathbf{x})=|y|_{\A_F}\sum_{\xi\in F^\times}a_f(\xi y\fd_F)W_{k}(\xi y_\infty)(\mathbf{x})\mathbf{e}_F(\xi x),
$$
where $y_\infty\in \C$ is the infinite part of $y$. 
\endenum

\end{prop}

\section{Approximate functional equation}\label{sp:Lvalue}

\subsection{Integral representation of special $L$-values} 

Let $f\in S_2(\fN,\chi)$ be a cusp form and $u\in\A_F^{(\infty)}$ a finite adele. 
Then we define the \emph{additively twisted L-function} $L(s,f,u)$ of $f$ by the following Dirichlet series:
\begin{align*}
 \frac{1}{|\O_F^\times|}\sum_{i=1}^{h_F}\sum_{\xi\in F^\times}\frac{a_f(\xi\fa_i\fd_F)\mathbf{e}_F(\xi a_i u)}{N(\xi\fa_i\fd_F)^s} \text{, converges absolutely for } \mathfrak{R}(s)>\frac{k+2}{2}.
\end{align*}
Clearly $L(s,f,0)$ is equal to the usual $L$-function $L(s,f)$ of $f$. 

In this section, we obtain an integral representation of $L(s,f,u)$ in a spirit of Hida \cite[Section 7]{hida1994critical}. Let us denote $t(u):=\begin{spmatrix} 1 & u \\ 0 & 1 \end{spmatrix}\in\GL_2(\A_F^{(\infty)})$ and $t_i:=\begin{spmatrix} a_i & 0 \\ 0 & 1 \end{spmatrix}$ for each $u\in\A_F^{(\infty)}$ and $i$ . 
We define the $m$-th component $f_m$ of a $L_{2k-2}(\C)$-valued function $f$ by 
$$
f(g)(\mathbf{x})=\sum_{m=0}^{2k-2} f_{m}(g) X^{2k-2-m}Y^{m} .
$$ 
Let us fix a branch of the square root function on $\C$. For $x\in\C$ and $y\in\C^\times$, define an element $(x;y)\in\SL_2(\C)$ by
$$
(x;y):=\begpmat y^{\frac{1}{2}} & x y^{-\frac{1}{2}}  \\ 0 & y^{-\frac{1}{2}} \endpmat.
$$
By Proposition \ref{four:whit:exp:hida}, we observe that 
\begin{align*}
f_{k-1} \big(t_i(0;y)t(u)\big)=\begin{spmatrix} 2k-2 \\ k-1 \end{spmatrix}|a_i|_{\A_F}\sum_{\xi\in F^\times}a_f(\xi\fa_i\fd_F)|y|^{k}K_{0}(4\pi|\xi y|)\mathbf{e}_F(\xi a_i u) .
\end{align*}
So the Mellin transform of the function $y\mapsto f_{k-1}\big(t_i(0;y)t(u)\big)$ on $\R^\times_{>0}$ is given by

\begin{align}\begin{split}\label{mellin:transform}
&|a_i|_{\A_F}^{s-1}\int_{0}^\infty f_{k-1} \big(t_i(0;y)t(u)\big)y^{2s-k} \frac{dy}{y} \\
=& \begin{spmatrix} 2k-2 \\ k-1 \end{spmatrix} |D_F|^s \sum_{\xi\in F^\times}\frac{a_f(\xi\fa_i\fd_F)\mathbf{e}_F(\xi a_i u)}{N(\xi\fa_i\fd_F)^s}\int_{0}^\infty y^{2s} K_{0}(4\pi y)\frac{dy}{y}
\end{split}\end{align} 
for $\text{Re}(s)>\frac{k+2}{2}$. 
By the following integration formula
\begin{equation*}
\int_0^\infty y^s{K_{j}(ay)}\frac{dy}{y}=\frac{1}{4}\Big(\frac{a}{2}\Big)^{-s}\Gamma\Big(\frac{s+j}{2}\Big)\Gamma\Big(\frac{s-j}{2}\Big)\text{ for }a>0\text{ and }\mathfrak{R}(s\pm j)>0 ,
\end{equation*}
the equation (\ref{mellin:transform}) becomes
\begin{align*}
|a_i|_{\A_F}^{s-1}\int_{0}^\infty f_{k-1} \big(t_i(0;y)t(u)\big)y^{2s-k} \frac{dy}{y} = \frac{\Gamma_{F,k}(s)}{|\O_F^\times|}\sum_{\xi\in F^\times}\frac{a_f(\xi\fa_i\fd_F)\mathbf{e}_F(\xi a_i u)}{N(\xi\fa_i\fd_F)^s}
\end{align*} 
for $\mathfrak{R}(s)>\frac{k+2}{2}$, where $\Gamma_{F,k}(s):=\begin{spmatrix} 2k-2 \\ k-1\end{spmatrix}\frac{|D_F|^s|\O_F^\times|}{4(2\pi)^{2s}}\Gamma(s)^2.$
Finally, we obtain the integral representation of $L(s,f,u)$:
\begin{align}\begin{split}\label{integral:rep}
\Gamma_{F,k}(s)L(s,f,u)=\sum_{i=1}^{h_F}|a_i|_{\A_F}^{s-1}\int_{0}^\infty f_{k-1} \big(t_i(0;y)t(u)\big)y^{2s-k}\frac{dy}{y}
\end{split}\end{align}
for $\mathfrak{R}(s)>\frac{k+2}{2}$.

\subsection{Functional equations}

Let $\omega_\chi:=\chi|\cdot|_{\A_F}^{k-2}$ be the standard normalization of $\chi$. Let $\fq$ be an integral ideal of $F$ coprime to $\fa_i^{-1}\mathfrak{d}_F\fN$ for each $i$. 
Let $R_\fq$ be the set of elements of $\prod_{v|\fq}F_v^\times$ corrensponding to a set of representatives of $(\fq^{-1}/\O_F)^\times$. Let us embed $u\in R_\fq$ into $\A_F^{(\infty)}$ by $u\mapsto(0,\cdots,0,u,0,\cdots,0)$.
Note that we have the following bijection:
$$
R_\fq\rightarrow (\widehat{\O}_F/\widehat{\fq})^\times ,\ u\mapsto\varpi_\fq u+\widehat{\fq}.
$$
For $u\in R_\fq$, let us define a \emph{completed L-function} $\Lambda(s,f,u)$ of $L(s,f,u)$ by
$$
\Lambda(s,f,u):=|\varpi_\fN\varpi_\fq^2|_{\A_F}^{-\frac{s}{2}}\Gamma_{F,k}(s)L(s,f,u).
$$
Let us define the \emph{Fricke involution} $W_\fN$ by taking the right translation with the element $\begin{spmatrix} 0 & -1 \\ \varpi_\fN & 0 \end{spmatrix}$ of $\GL_2(\A_F^{(\infty)})$:
$$
W_{\fN}f(g)(\mathbf{x}):=(-1)^{k-1}|\varpi_\fN|_{\A_F}^{\frac{k}{2}-1}\omega_\chi^{-1}\big(\det(g)\big)f\bigg(g\begpmat 0 & -1 \\ \varpi_\fN & 0 \endpmat \bigg)(\mathbf{x}).
$$
Then $W_{\fN}f$ is a cusp form in $S_k(\fN,\ovl{\chi})$. Note that for $u\in R_\fq$, we can consider $u\varpi_{\fq}$ to be an element of $(\O_{F,\fq}/\fq^2)^\times$. For $u\in R_\fq$, let us denote by $u^\p\varpi_{\fq}$ a representative of the mod $\fq^2$ inverse of $-u\varpi_{\fq}$. Then we obtain a functional equation of $\Lambda(s,f,u)$ as follows:

\begin{prop}\label{func:eq} For $u\in R_\fq$, we have the following functional equation:
\begin{equation}\label{functional:eq}
\Lambda(s,f,u)=\omega_{\chi}(\varpi_\fq)\Lambda(k-s,W_{\fN}f,u^\p).
\end{equation}
Also we have the anlaytic continuation of $\Lambda(s,f,u)$ to the entire $s\in\C$.
\end{prop}

\begin{proof}
Let us denote $\fQ=\fN\fq^2$. Then there exists a $h_F$-permutation $\sigma$ and elements $\xi_i\in F^\times$ such that $\fa_i^{-1}\fQ=\xi_i\fa_{\sigma(i)}$ for each $i$. For $\xi\in F^\times$, let us denote $w_{\xi}:=\xi^{(\infty)}\varpi_{\xi\O_F}^{-1}$. Then we obtain 
$$
w_{\xi_i}\in\widehat{\O}_F^\times,\ (w_{\xi_i})_{\fq}=1,\ a_i^{-1}\varpi_\fQ=a_{\sigma(i)}\xi_i^{(\infty)}w_{\xi_i}^{-1},\text{ and }|\xi_i|=|a_i^{2}a_{\sigma(i)}^{2}\varpi_\fQ^{-2}|_{\A_F}^{1/4}.
$$ 
Let us denote by $f_{(i),j}^u$ the function $y\mapsto f_{j}\big(t_i(0;y)t(u)\big)$ on $\R_{>0}^\times$. Then by the equation (\ref{integral:rep}), we obtain
\begin{align}\begin{split}\label{add:comp:lftn:int:rep}
\Lambda(s,f,u)=\sum_{i=1}^{h_F}\frac{|\varpi_\fQ|_{\A_F}^{-\frac{s}{2}}}{|a_i|_{\A_F}^{1-s}}\bigg(\int_{|a_i^{-2}\varpi_\fQ|_{\A_F}^{1/4}}^\infty +\int_{0}^{|a_i^{-2}\varpi_\fQ|_{\A_F}^{1/4}}\bigg) f_{(i),k-1}^u(y)y^{2s-k} \frac{dy}{y}. 
\end{split}\end{align}

Let us set $\gamma_i=\begin{spmatrix} 0 & -1 \\ \xi_i & 0 \end{spmatrix}\in\GL_2(F)$. Let $b(u)\in\widehat{\O}_F$ be the element such that $b(u)_\fq\varpi_\fq^2-(u\varpi_\fq)(u^\p\varpi_\fq)=1$ and $b(u)_v=1$ for $v$ coprime to $\fq$. 
Under the isomorphism $\GL_2(\A_F)\cong\GL_2(\C)\times\GL_2(\A_F^{(\infty)})$, $\gamma_i t_i(0;y)t(u)$ can be decomposed into
\begin{align*}
\gamma_i t_i(0;y)t(u)=\big(\xi_{i}^{1/2}(0;|\xi_i|^{-1}y^{-1})U_i,\varpi_\fq a_{\sigma(i)}^{-1}t_{\sigma(i)}t(u^\p) A_\fN B_u C_i\big)
\end{align*}
for each $y\in \C^\times$ and $i$, where $U_i=(0;|\xi_i|\xi_i^{-1})\begin{spmatrix} 0 & -1 \\ 1 & 0 \end{spmatrix}\in\SU_2(\C)$, $A_\fN=\begin{spmatrix} 0 & -1 \\ \varpi_\fN & 0 \end{spmatrix}$,
$$
B_u=\begin{pmatrix} \varpi_\fq & u\varpi_\fq \\ u^\p\varpi_\fq & b(u)\varpi_\fq \end{pmatrix}\in U_1(\fN),\text{ and } C_i=\begin{pmatrix} w_{\xi_i} & 0 \\ 0 & 1 \end{pmatrix}\in U_1(\fN).
$$
Note that $\GL_2(F)$ acts on $\GL_2(\A_F)\cong\GL_2(\C)\times\GL_2(\A_F^{(\infty)})$ diagonally.
By using the above equations and the properties of $f$, we have
\begin{align*}
f_{(i),k-1}^u(y)
&=f_{k-1}\big(t_i(0;y)t(u)\big)=f_{k-1}\big(\gamma_i t_i(0;y)t(u)\big) \\ 
&=|\xi_i|^{2-k}\chi_{\fN}(B_u C_i)\chi(\varpi_\fq a_{\sigma(i)}^{-1}) \Big(f\big(t_{\sigma(i)}(0;|\xi_i|^{-1}y^{-1})t(u^\p)A_\fN\big)\big(U_i\mathbf{x}\big)\Big)_{k-1} \\
&=|\xi_i|^{2-k}|a_{\sigma(i)}|_{\A_F}^{k-2}|\varpi_\fQ|_{\A_F}^{1-\frac{k}{2}}\omega_{\chi}(\varpi_\fq)(W_\fN f)^{u^\p}_{(\sigma(i)),k-1}(|\xi_i|^{-1}y^{-1}).
\end{align*}
By using the above equations, we can write the second term of the integral (\ref{add:comp:lftn:int:rep}) as 
\begin{align*}
\sum_{i=1}^{h_F}\frac{\omega_{\chi}(\varpi_\fq)|\varpi_\fQ|_{\A_F}^{\frac{s-k}{2}}}{|a_{\sigma(i)}|^{1-k+s}_{\A_F}}\int_0^{|a_i^{-2}\varpi_\fQ|_{\A_F}^{1/4}}(W_\fN f)^{u^\p}_{(\sigma(i)),k-1}\big(|\xi_i|^{-1}y^{-1}\big)(|\xi_i|y)^{2s-k}\frac{dy}{y}.
\end{align*}
By applying the change of variable to the above equation, we can rewrite the equation (\ref{add:comp:lftn:int:rep}) as follows:
\begin{align}\begin{split}\label{add:comp:lftn:int:rep:2}
\Lambda(s,f,u)=&\sum_{i=1}^{h_F}\bigg(\int_{|a_i^{-2}\varpi_\fQ|_{\A_F}^{1/4}}^\infty r_i(s)f^u_{(i),k-1}(y)y^{2s-k}\frac{dy}{y} \\
&+\int_{|a_i^{-2}\varpi_\fQ|_{\A_F}^{1/4}}^\infty\omega_{\chi}(\varpi_\fq)r_i(k-s)(W_\fN f)^{u^\p}_{(i),k-1}(y)y^{k-2s}\frac{dy}{y}\bigg),
\end{split}\end{align}
where $r_i(s)=|\varpi_\fQ|_{\A_F}^{-\frac{s}{2}}|a_i|_{\A_F}^{s-1}$. Note that $W_\fN f\in S_2(\fN,\ovl{\chi})$, $W_\fN^2 f=f$ for $f\in S_2(\fN,\chi)$, and $(u^\p)^\p=u$.
By the straightforward computation, we can obtain the desired funtional equation from the integral representation (\ref{add:comp:lftn:int:rep:2}).
Also the equation (\ref{add:comp:lftn:int:rep:2}) implies that $\Lambda(s,f,u)$ has the analytic continuation to the entire $s\in\C$.
\end{proof}

By using our functional equation $(\ref{func:eq})$, we obtain an approximate functional equation of $\Lambda(s,f,u)$ for $u\in R_\fq$ in a spirit of Luo-Ramakrishnan \cite{LuoRama}. Let us fix an infinitely differentiable function $\Phi$ 
on $\R_{>0}^\times$ with compact support such that $\int_0^\infty\Phi(y)\frac{dy}{y}=1$. Define functions
\begin{align*}
V_{1,s}(x)&:=\frac{1}{2\pi i}\int_{2-i\infty}^{2+i\infty}\kappa(t)\Gamma_{F,k}(s+t) x^{-t}\frac{dt}{t},\\
V_{2,s}(x)&:=\frac{1}{2\pi i}\int_{2-i\infty}^{2+i\infty}\kappa(-t)\Gamma_{F,k}(s+t) x^{-t}\frac{dt}{t},
\end{align*}
where $\kappa(t)=\int_0^\infty\Phi(y)y^t\frac{dy}{y}$. By shifting the contour, one can show that for $s\neq 1/2$, $V_{1,s}$ and $V_{2,s}$ satisfy following:
\begin{align}\begin{split}\label{aux.func.esti}
V_{j,s}(x)&=O(\Gamma_{F,k}(\mathfrak{R}(s)+j)x^{-j})\text{ for all }j\geq 1 \text{ as }x\rightarrow\infty, \\
V_{j,s}(x)&=\Gamma_{F,k}(s)+O\Big(\Gamma_{F,k}\Big(\mathfrak{R}(s)-\frac{1}{2}\Big)x^{\frac{1}{2}}\Big) \text{ as }x\rightarrow 0,
\end{split}\end{align}
where the implicit constants depend only on $j$ and $\Phi$. Let $u$ be an element of $R_\fq$. Let us recall that $\fq$ is coprime to $\fa_i^{-1}\fN$ for each $i$. For $\mathfrak{R}(s)>\frac{k+2}{2}$ and $y>0$, we obtain
\begin{align}\begin{split}\label{series:exp}
&\frac{1}{2\pi i}\int_{2-i\infty}^{2+i\infty}\kappa\big((-1)^{j-1}t\big)\Gamma_{F,k}(s+t) L(s+t,f,u) y^t\frac{dt}{t}\\
=&\frac{1}{|\O_F^\times|}\sum_{i=1}^{h_F}\sum_{\xi\in F^\times}\frac{a_f(\xi\fa_i\fd_F)\mathbf{e}_F(\xi a_i u)}{N(\xi\fa_i\fd_F)^s}V_{j,s}\Big(\frac{N(\xi\fa_i\fd_F)}{y}\Big) .
\end{split}\end{align}
by the definition of $L(s,f,u)$. From the Residue theorem, we have
\begin{align*}
\Gamma_{F,k}(s) L(s,f,u)&=\frac{1}{2\pi i}\int_{2-i\infty}^{2+i\infty}\kappa(t)\Gamma_{F,k}(s+t) L(s+t,f,u)y^t\frac{dt}{t} \\
&+\frac{1}{2\pi i}\int_{-2+i\infty}^{-2-i\infty}\kappa(t)\Gamma_{F,k}(s+t) L(s+t,f,u)y^t\frac{dt}{t} .
\end{align*}
Putting the equations (\ref{functional:eq}), (\ref{series:exp}) in the above equation and applying Proposition \ref{four:whit:exp:hida}.(1), we obtain
\begin{align}\begin{split}\label{approx:func:eq}
\Lambda(s,f,u)&=\sum_{i=1}^{h_F}\bigg(N(\fN\fq^2)^{\frac{s}{2}}\sum_{0\neq\a\in\fa_i^{-1}}\frac{a_f(\a\fa_i)\mathbf{e}_F(\a u)}{N(\a\fa_i)^s|\O_F^\times|}V_{1,s}\Big(\frac{N(\a\fa_i)}{y}\Big) \\
+\omega_{\chi}(\varpi_\fq)&N(\fN\fq^2)^{\frac{k-s}{2}}\sum_{0\neq\a\in\fa_i^{-1}}\frac{a_{{W_{\fN}} f}(\a\fa_i)\mathbf{e}_F(\a u^\p)}{N(\a\fa_i)^{k-s}|\O_F^\times|}V_{2,k-s}\Big(\frac{N(\a\fa_i)y}{N(\fN\fq^2)}\Big)\bigg).
\end{split}\end{align}
Note that $\mathbf{e}_F(\alpha u)=\mathbf{e}_F(\alpha a_i d_F^{-1} u)$ as $\fq$ is coprime to $\fa_i^{-1}\mathfrak{d}_F$ for each $i$.

\subsection{Twisted cusp forms}\label{twisted:cuspforms}
Let $f\in S_k(\fN,\chi)$ be a cusp form. For a finite order Hecke character $\phi:F^\times\backslash\A_F^\times\rightarrow\C^\times$ of conductor $\fq$, the \emph{twisted cusp form} $f\otimes\phi:\text{GL}_2(\A_F)\rightarrow L(2k-2,\C)$ is defined by
\begin{equation}\label{twisted:cuspform:decomp:eq}
f\otimes\phi(g)(\mathbf{x}):=G(\phi)^{-1}\phi\big(\det(g)\big)\sum_{u\in R_{\fq}}\phi(\varpi_\fq u)f\big(gt(u)\big)(\mathbf{x}),
\end{equation} 
where $G(\phi)$ is the Gauss sum of $\phi$ (see Hida \cite[Section 6]{hida1994critical}). 
Then we have $f\otimes\phi\in S_{k}(\fN\cap\fq^2,\chi\phi^2)$ and $a_{f\otimes\phi}(\fa)=a_f(\fa)\phi(\fa)$, where
\begin{displaymath}
\phi(\fa) := \left\{ \begin{array}{ll}
\phi(\varpi_\fa) & \textrm{if $\fa$ is coprime to $\mathfrak{f}(\phi)$},\\
0 & \textrm{otherwise}.
\end{array} \right.
\end{displaymath} 
From the equation (\ref{integral:rep}), the special value $L(\frac{k}{2},f\otimes\phi)$ of the $L$-function of $f\otimes\phi$ is given by 
\begin{equation}\label{special:value:of:twisted:modular:L:functions}
\Gamma_{F,k}\Big(\frac{k}{2}\Big)L\Big(\frac{k}{2},f\otimes\phi\Big)=\sum_{i=1}^{h_F}\sum_{u\in R_{\fq}}\frac{\phi(a_i)\phi(\varpi_\fq u)}{G(\phi)N(\fa_i)^{\frac{k}{2}-1}}\int_0^\infty f_{k-1} \big(t_i(0;y)t(u)\big)\frac{dy}{y}.
\end{equation}

\section{Modular symbols and Eichler-Shimura isomorphism}\label{section:conj:on:mod:symb}
In this section, we will give the definition and the properties of modular symbols. Also we discuss the homology and cohomology groups of Bianchi modular manifolds. To associate the cusp forms and the cohomology classes, we introduce the generalized Eichler-Shimura isomorphism.

\subsection{Homology and Cohomology of Bianchi modular manifolds}\label{alg:co:hom:mod:sym} 
We will introduce the Bianchi modular manifold and study its homology and cohomology. For the case of classical modular curves, see Ash-Stevens \cite{AshStevens} and Manin \cite{Manin}. 

Let us define the hyperbolic upper half space by $\mathscr{H}_3:=\C\times\R_{>0}$. Then there is a left action of $\SL_2(\C)$ on $\mathscr{H}_3$, which is defined by
\begin{equation*}
\begin{pmatrix} a & b \\ c & d \end{pmatrix}\cdot (x,y):=\bigg(\frac{(ax+b)\ovl{(cx+d)}+a\ovl{c}y^2}{|cx+d|^2+|c|^2y^2},\frac{y}{|cx+d|^2+|c|^2y^2}\bigg) .
\end{equation*}
Note that the element $(x;y)\in\SL_2(\C)$ for $x\in\C$ and $y\in\R_{>0}^\times$ defined in the Section \ref{sp:Lvalue} corresponds to $(x,y)\in\mathscr{H}_3$ under the following bijection:
$$
\SL_2(\C)/\SU_2(\C)\rightarrow\mathscr{H}_3,\ g\SU_2(\C)\mapsto g\cdot (0,1).
$$

Let $R$ be a Dedekind domain, $\Gamma$ a congruence subgroup of $\SL_2(\O_F)$ such that the orders of the torsion elements of $\Gamma$ are invertible in $R$. 
Let us define the Bianchi modular manifold $Y_\Gamma$ of $\Gamma$ by $Y_\Gamma:=\Gamma\backslash\mathscr{H}_3$. 
Let us denote $\ovl{\mathscr{H}_3}$ by the space obtained by attaching $\P^1(F)$ to $\mathscr{H}_3$ via the map $\a\mapsto(\a,0)$ and $\infty\mapsto(0,\infty)$. Let us set $X_\Gamma:=\Gamma\backslash\ovl{\mathscr{H}_3}$. Then $X_\Gamma$ is the Satake compactifiation of $Y_\Gamma$.
Let us denote the set of cusps of $X_\Gamma$ by $C_\Gamma:=\Gamma\backslash\P^1(F)$. From the relative homology sequence of a pair $(X_\Gamma,C_\Gamma)$, we can obtain the following exact sequence:
\begin{equation}\label{hom:mod:symb}
0 \rightarrow H_1(X_\Gamma,R) \rightarrow H_1(X_\Gamma,C_\Gamma,R) \rightarrow H_0(C_\Gamma,R) \rightarrow H_0(X_\Gamma,R) \rightarrow  0.
\end{equation}
Note that $H_1(C_\Gamma,R)=H_0(X_\Gamma,C_\Gamma,R)=0$.

Let $Y^*_\Gamma$ be the Borel-Serre compactification of $Y_\Gamma$ (see Borel-Serre \cite{borel1973corners}). 
Note that the boundary $\partial Y_\Gamma^*$ of $Y_\Gamma^*$ is a disjoint union of $|C_\Gamma|$ $2$-tori. Also note that there is a homotopy equivalence $Y_\Gamma\cong Y_\Gamma^*$. 
Then we have the following boundary exact sequence:
\begin{equation}\label{cohom:mod:symb:5}
\begin{tikzcd} \cdots \arrow{r} & H^i(Y_\Gamma,R) \arrow{r} & H^i(\partial Y^*_\Gamma,R) \arrow{r} & H^{i+1}_{\operatorname{c}}(Y_\Gamma,R) \arrow{r} & \cdots \end{tikzcd}.
\end{equation}
Breaking the sequence (\ref{cohom:mod:symb:5}), we obtain the following exact sequence:
\begin{align}
0 \rightarrow H^0(Y_\Gamma,R) \rightarrow H^0(\partial Y_\Gamma^*,R) \rightarrow H^1_{\operatorname{c}}(Y_\Gamma,R) \rightarrow H^1_{\text{par}}(Y_\Gamma,R) \rightarrow 0 \label{cohom:mod:symb},
\end{align}
where $H^1_{\text{par}}(Y_\Gamma,R)$ is the image of the natural map $H^1_{\operatorname{c}}(Y_\Gamma,R)\rightarrow H^1(Y_\Gamma,R)$. Note that $H^0(Y_\Gamma,R)\cong R$ and $H^0(\partial Y_\Gamma,R)\cong \oplus_s R^{\Gamma_{s}}$, where $s$ runs over $C_\Gamma$.

We have the isomorphism $H^1_{\text{par}}(Y_\Gamma,R)\cong H^1(X_\Gamma,R)$ induced by the injection $Y^*_\Gamma\rightarrow X_\Gamma$. Also we have $H^2(Y_\Gamma,R)\cong H_1(X_\Gamma,C_\Gamma,R)$
by the Lefschetz-Poincar\'{e} duality. Then from the Poincar\'{e} duality, we can induce non-degenerate $R$-bilinear pairings as follows:
\begin{align}
\cap:&H_1(X_\Gamma,R)^\p\times H^1_{\text{par}}(Y_\Gamma,R)^\p \rightarrow R, \label{lefschetz:poincare:pairing} \\
\cap:&H_1(X_\Gamma,C_\Gamma,R)^\p\times H^1_{\operatorname{c}}(Y_\Gamma,R)^\p \rightarrow R, \label{lefschetz:poincare:pairing:2}
\end{align}
where $H^\p$ is the torsion free part of a finitely generated $R$-module $H$. 

\begin{rem}\label{pairings:compatible} 
Note that there is a Hecke equivariant section $\iota$ of the map 
$$
H^1_{\operatorname{c}}(Y_\Gamma,\C)\rightarrow H^1_{\text{par}}(Y_\Gamma,\C)\cong H^1(X_\Gamma,\C)
$$ 
(see Hida \cite[Section 2]{hida1999non}). Also note that for a parabolic cohomology class $\delta\in H^1_{\text{par}}(Y_\Gamma,\C)$, $\delta-\iota(\delta)$ is an exact form. Thus we have 
\begin{equation}\label{pairing:equality}
c\cap\delta=c\cap\iota(\delta)
\end{equation}
for each homology class $c\in H_1(X_\Gamma,\Z)$. 

On the other hand, choose a relative homology class $\xi\in H_1(X_\Gamma,C_\Gamma,\Z)$. By the generalized Manin-Drinfeld theorem (Kur{\v c}anov \cite[Lemma 2, Theorem 2]{kurcanov1978cohomology}), $N_\xi\xi\in H_1(X_\Gamma,\Z)$ for some non-zero integer $N_\xi$. Thus we have $N_\xi\xi\cap\delta=N_\xi\xi\cap\iota(\delta)$. Therefore, if $R=\C$, then the pairings (\ref{lefschetz:poincare:pairing}) and (\ref{lefschetz:poincare:pairing:2}) are compatible under the map $\iota$ and the sequence (\ref{hom:mod:symb}).
\end{rem}

\subsection{Modular symbols of weight (2,2)}
Let $\Gamma$ be a torsion free congruence subgroup of $\SL_2(\O_F)$. For $\a,\beta\in \P^1(F)$, let $\{\a,\beta\}$ be the geodesic from $\a$ to $\beta$ in $\ovl{\mathscr{H}_3}$. Let $\{\a,\beta\}_\Gamma$ be the image of $\{\a,\beta\}$ in $X_\Gamma$. Then we can consider $\{\a,\beta\}_\Gamma$ to be an element of $H_1(X_\Gamma,C_\Gamma,\Z)$. The classes $\{\a,\beta\}_\Gamma$ are called \emph{modular symbols} of level $\Gamma$ and weight $(2,2)$. We have the following relations on the modular symbols:
\begin{itemize}
\item[(1)]
$\{\a,\a\}_\Gamma=0$ in $H_1(X_\Gamma,C_\Gamma,\Z)$ for $\a\in\P^1(F)$.
\item[(2)]
$\{\a_1,\a_2\}_\Gamma=\{\a_1,\a_3\}_\Gamma+\{\a_3,\a_2\}_\Gamma$ for $\a_1,\a_2,\a_3\in\P^1(F)$.
\item[(3)]
$\{\a,\beta\}_\Gamma=\{\gamma\a,\gamma\beta\}_\Gamma$ for $\a,\beta\in\P^1(F)$ and $\gamma\in\Gamma$. 
\end{itemize}
By the definition of the relative homology group, $H_1(X_\Gamma,C_\Gamma,\Z)$ is generated by the symbols $\{\a,\beta\}_\Gamma$. 
From the exact sequence (\ref{hom:mod:symb}), we have
$\ker(\partial)=H_1(X_\Gamma,\Z)$, where $\partial$ is the boundary map defined by 
$$
\partial:H_1(X_\Gamma,C_\Gamma,\Z)\rightarrow H_0(C_\Gamma,\Z),\ \{\a,\beta\}_\Gamma\mapsto\{\a\}_\Gamma-\{\beta\}_\Gamma.
$$ 
Note that we have the following map:
$$
\text{pr}_\Gamma:\Gamma\rightarrow H_1(X_\Gamma,\Z),\ \gamma\mapsto\{0,\gamma\cdot 0\}_\Gamma.
$$
Then $\text{pr}_\Gamma$ is a surjective group homomorphism, which is trivial on the parabolic subgroups $\Gamma_s$ of $\Gamma$.

\subsection{Special modular symbols}\label{sp:mod:symb}
For each $1\leq i\leq h_F$, let us denote 
$$
\Gamma_1^i(\fN):=\GL_2(F)\cap t_i\GL_2(\C)U_1(\fN)t_i^{-1},\ \ovl{\Gamma_1^i(\fN)}:=\SL_2(F)\cap\Gamma_1^i(\fN),
$$ 
which are congruence subgroups of $\GL_2(\O_F)$ and $\SL_2(\O_F)$, respectively.
Since $[\Z:\fN\cap\Z]>3$, we observe that $\ovl{\Gamma_1^i(\fN)}$ has no torsion element (see Urban \cite[Lemme 2.3.1]{urban1995formes}). Therefore, for $\Gamma=\ovl{\Gamma_1^i(\fN)}$, $Y_\Gamma$ is a manifold, not an orbifold.
Let us denote $Y_1^i(\fN):=Y_\Gamma$ and $X_1^i(\fN):=X_\Gamma$ for $\Gamma=\ovl{\Gamma_1^i(\fN)}$. Let us define the Bianchi modular manifold $Y_1(\fN)$ of level $U_1(\fN)$ by 
$$
Y_1(\fN):=\GL_2(F)\backslash\GL_2(\A_F)/\C^\times\operatorname{U}_2(\C)U_1(\fN).
$$
Let us recall that $\fa_i^{-1}$ is coprime to $\fN$ for each $i$. Then by the strong approximation theorem on $\GL (2)$, we obtain the following isomorphism:
$$
Y_1(\fN)\rightarrow\coprod_{i=1}^{h_F}Y_1^{i}(\fN),\ g=\gamma t_j g^\p h\mapsto \ovl{\Gamma_1^j(\fN)}\frac{g^\p}{\det(g^\p)}\operatorname{SU}_2(\C)\in Y_1^{j}(\fN),
$$
where $\gamma\in\GL_2(F)$, $g^\p\in\GL_2(\C)$, $h\in U_1(\fN)$ and $1\leq j\leq h_F$. 
Note that 
$$
H_1(X_1(\fN),\Z)\cong\bigoplus_{i=1}^{h_F}H_1(X_1^i(\fN),\Z).
$$

Let $\fq$ be an integral ideal of $F$. For $u\in R_\fq$ and $i$, consider the following set:
$$
\{t_i(0;y)t(u):y\in\R^\times_{>0}\}\subset\GL_2(\A_F).
$$
Then we have the following proposition:

\begin{prop}\label{adelic:modular:symbol:reprn} Assume that $\fq$ is coprime to $\fa_i^{-1}$. Then for each $u\in R_\fq$, there is an element $\a_u=\begin{spmatrix} 1 & r_u \\ 0 & 1 \end{spmatrix}\in \SL_2(F)$ and $c^{(i)}_u\in\widehat{\O}_F$ such that $(c^{(i)}_u)_\fq=0$ and
$$
\a_{u}t_i g t(u)=t_i \a_{u,\infty} g t(c^{(i)}_u)
$$
in $\GL_2(\A_F)$ for any $g\in \GL_2(\C)$. 
\end{prop}

\begin{proof} Choose an element $r_u$ of $(\fq^{-1}/\O_F)^\times$ corresponding to the element $-\varpi_{\fq} u \in (\widehat{\O}_F/\widehat{\fq})^\times$
under the following well-defined bijection: 
$$
(\fq^{-1}/\O_F)^\times\rightarrow (\widehat{\O}_F/\widehat{\fq})^\times,\ r+\O_F\mapsto \varpi_{\fq}r_{\fq}+\widehat{\fq}.
$$
Then we have the equation as follows:
$$
\begin{pmatrix} 1 & r_u \\ 0 & 1 \end{pmatrix}t_i(0;y)t(u)=t_i\begin{pmatrix} 1 & r_u \\ 0 & 1 \end{pmatrix}_\infty(0;y)\begin{pmatrix} 1 & u+r_u^{(\infty)}a_i^{-1} \\ 0 & 1 \end{pmatrix}.
$$
Note that $(u+r_u^{(\infty)} a_i^{-1})_v=0$ for each $v|\fq$, and $(u+r_u^{(\infty)} a_i^{-1})_v\in\O_v$ for each $v$ coprime to $\fq$. So we conclude the proof. 
\end{proof}

\begin{rem} Note that we can observe that $\a_u$ and $r_u$ does not depend on $i$. For $u\in R_\fq$, the set $\{t_i(0;y)t(u):y\in\R^\times_{>0}\}$ represents the homology class $\{\infty,r_{u}\}_{\ovl{\Gamma^i_1(\fN)}}=\a_{u,\infty}\{\infty,0\}_{\ovl{\Gamma^i_1(\fN)}}$ by the Proposition \ref{adelic:modular:symbol:reprn}. So from now on, let us denote by $\xi^{(i)}_u$ the class $\{\infty,r_{u}\}_{\ovl{\Gamma^i_1(\fN)}}$ for $u\in R_\fq$ if $\fq$ is coprime to $\fa_i^{-1}$. 
\end{rem}

Let $\mathfrak{l}$ be a prime ideal of $F$ 
coprime to $\fa_i^{-1}$ for each $i$.
For an integer $n>0$ and an element $a\in\O_{F,\fl}^\times$, let us define the homology class $\xi^{(i)}_{\fl,n}(a)$ by
$$
\xi^{(i)}_{n,\fl}(a):=\xi_{a/\varpi_\fl^{n}}^{(i)}\in H_1(X_1^i(\fN),C_1^i(\fN),\Z),
$$
or simply denote by $\xi^{(i)}_{n}(a)$ if there is no ambiguity on $\fl$. 
We consider the submodule $M_{n,\fl}$ of $H_1(X_1(\fN),\Z)$, which is defined by
$$
M_{n,\fl}:=\bigoplus_{i=1}^{h_F}\langle\xi_n^{(i)}(a):a\in \O_{F,\fl}^\times\rangle\cap H_1(X_1(\fN),\Z).
$$
Note that $H_1(X_1(\fN),\Z)$ is a finite generated abelian group since $X_1(\fN)$ is a finite dimensional compact manifold.
In section \ref{section:full:rank}, we will show that $M_{n,\fl}$ is of full-rank in $H_1(X_1(\fN),\Z)$ for prime ideals $\fl$ of sufficiently large norm $N(\fl)$.

\subsection{Generalized Eichler-Shimura isomorphism}\label{Eich:shi:isom:poin:dual}
Let us follow Hida \cite{hida1994critical} to define the generalized Eichler-Shimura map. Let $S_2\big(\Gamma_1^i(\fN)\big)$ be the space of cusp forms on $\GL_2(\C)$ of weight $(2,2)$ and level $\Gamma_1^i(\fN)$ (for the definition, see Hida \cite[Section 2]{hida1994critical}). For $f\in S_2\big(\Gamma_1^i(\fN)\big)$, define a function $f^\p$ on $\SL_2(\C)$ by 
$$
f^\p(g,\mathbf{a}):=\begin{pmatrix} f_0(g) & f_1(g) & f_2(g) \end{pmatrix}\Psi\big({}^t j\big(g,(0,1)\big)\mathbf{a}\big),
$$
where $\mathbf{a}=\begin{spmatrix} A \\ B \end{spmatrix}$, $\Psi(\mathbf{a})={}^t\begin{pmatrix} A^2 & AB & B^2 \end{pmatrix}$ and
$$
j\bigg(\begin{pmatrix} a & b \\ c & d \end{pmatrix},(x,y)\bigg):=\begin{pmatrix} cx+d & -cy \\ \ovl{c}y & \ovl{cx+d} \end{pmatrix}
$$
for $(x,y)\in\mathscr{H}_3$.
Let us define a differential $1$-form $\delta^{(i)}(f)$ on $Y_1^i(\fN)$ by 
$$
\delta^{(i)}(f)\big((x,y)\big):=f^\p\big((x;y),\mathbf{a}\big)|_{A^2=dx,\ AB=-dy,\ B^2=-d\ovl{x}},
$$
where $x,y,\ovl{x}$ are the local coordinates on $Y_\Gamma$. Note that we have following: 
$$
\ovl{\Gamma_1^i(\fN)}\backslash\SL_2(\C)/\SU_2(\C)\cong Y_1^i(\fN),\ \ovl{\Gamma_1^i(\fN)} g\SU_2(\C)\mapsto\ovl{\Gamma_1^i(\fN)}\big(g\cdot(0,1)\big).
$$ 
Then we obtain the \emph{generalized Eichler-Shimura isomorphism}:
\begin{prop}\label{eichler:shimura:harder} For $1\leq i\leq h_F$, the map $\delta^{(i)}$ is a Hecke equivariant isomorphism of $\C$-vector spaces:
$$
\delta^{(i)}:S_2\big(\Gamma_1^i(\fN)\big)\rightarrow H^1_{\operatorname{par}}(Y_1^i(\fN),\C)\cong H^1(Y_1^i(\fN),\C).
$$
\end{prop}
\begin{proof} By Hida \cite[Corollary 2.2, Proposition 3.1]{hida1994critical}, we are done.
\end{proof}
Note that for $(x,y)\in Y_1^i(\fN)$, we have
\begin{equation}\label{eichler:shimura:explicit} \delta^{(i)}(f)\big((x,y)\big)=f_0\big((x;y)\big)\frac{dx}{y}-f_1\big((x;y)\big)\frac{dy}{y}-f_2\big((x;y)\big)\frac{d\ovl{x}}{y}.
\end{equation}

Let us denote by $S_2\big(U_1(\fN)\big)$ the space of cusp forms on $\GL_2(\A_F)$ of weight $(2,2)$ and level $U_1(\fN)$ (for the definition, see Hida \cite[Section 3]{hida1994critical}). For $f\in S_2\big(U_1(\fN)\big)$, let $f_{(i)}$ be the function on $\GL_2(\C)$ defined by $f_{(i)}(g):=f(t_i g)$. Then we have the following isomorphism:
$$
S_2\big(U_1(\fN)\big)\cong\bigoplus_{i=1}^{h_F}S_2\big(\Gamma_1^i(\fN)\big),\ f\mapsto (f_{(i)})_{i=1}^{h_F}.
$$ 
Then for a cusp form $f\in S_2\big(U_1(\fN)\big)$, we have
\begin{equation}\label{integral:evaluation} 
\xi_{u}^{(i)}\cap\iota\big(\delta^{(i)}(f_{(i)})\big)=\int_0^\infty f_1\big(t_i(r_{u};y)\big)\frac{dy}{y}=\int_0^\infty f_1\big(t_i(0;y)t(u)\big)\frac{dy}{y}
\end{equation}
by Proposition \ref{adelic:modular:symbol:reprn} and the equation (\ref{eichler:shimura:explicit}).

\section{Lattice estimation}
In this section, we estimate the number of the elements in the lattices of bounded norm. Also we find the elements of the smallest norm in certain lattices.

Let $n>0$ be an integer and $\mathfrak{l}$ a prime ideal of $F$. Then we have the following lemma:
\begin{lem}\label{lattice:size:esti} For each $\beta\in\O_F$, we have 
\begin{align*}
\#\{\a\in\beta+\fl^{n}:\ N(\a)\leq x\} \ll_{F}\max\Big\{\frac{x}{N(\fl)^n},1\Big\}.
\end{align*}
\end{lem}
\begin{proof}
By Kwon-Sun \cite[Proposition 5.2]{kwon2020nonvanishing}, we can conclude the proof since the size of the unit group $\O_F^\times$ is finite. Note that \cite[Proposition 5.2]{kwon2020nonvanishing} comes from Rohrlich \cite{Rohrlich}, which is about the coherent cone decomposition of lattices in the Minkowski space of $F$.
\end{proof}


For $\beta\in\O_{F}$, let $[\beta]_{n,\fl}$ be the set consists of the elements $\a\in\beta+\fl^{n}$ such that 
$$
N(\a)=\min_{\a_1\in\beta+\fl^{n}}N(\a_1).
$$ 
We show that $[\beta]_{n,\fl}$ is a singleton set for sufficiently large $N(\fl)^n$:


\begin{lem}\label{suff:large:n:identity}
Let $\beta\in\O_F$ such that $4N(\beta)< N(\fl)^{n}$. Then we have $[\beta]_{n,\fl}=\{\beta\}$. 
\end{lem}
\begin{proof}
By the definition, the set $[\beta]_{n,\fl}$ is non-empty. Thus we can choose an element $\beta_1\in[\beta]_{n,\fl}$. Let us assume that $\beta_1\neq\beta$. Note that we have $0\neq\beta_1-\beta\in\fl^{n}$ and $N(\beta_1)\leq N(\beta)$ by the definition of $[\beta]_{n,\fl}$.
Then by the triangle inequality, we obtain
$$
N(\fl)^{n}\leq N(\beta-\beta_1)\leq (N(\beta)^\frac{1}{2}+N(\beta_1)^\frac{1}{2})^2\leq 4N(\beta) < N(\fl)^{n},
$$ 
which is a contradiction. Therefore, $[\beta]_{n,\fl}=\{\beta\}$, which concludes the proof.
\end{proof}


\section{Generation of homology groups}\label{section:full:rank}


\subsection{Estimation of integrals along special modular symbols}

Let $n>0$ be an integer and $\fl$ a prime ideal of $F$ whose norm $N(\fl)$ is coprime to $\fa_i^{-1}\fN$ for each $i$. Let us denote $\ell:=N(\fl)$.
In this section, we prove that the module $M_{n,\fl}$ defined in Section \ref{section:conj:on:mod:symb} is of full-rank in $H_1(X_1(\fN),\Z)$ for sufficiently large $\ell^n$. Before proving the generation result, we need the following lemma about the estimation on the Kloosterman-like exponential sum:

\begin{lem}\label{exp:sum} Le $v>0$ be an integer such that $v\leq\lceil\frac{n}{2}\rceil$. For $b,c\in\O_{F,\fl}$ with $v\leq n-\lceil\frac{n}{2}\rceil-\min\{v_\fl(b),v_\fl(c)\}$, we have
\begin{equation}\label{exp:sum:esti}
\sum_{a\equiv a_0 (\fl^v)}\mathbf{e}_F\bigg(\frac{ab+a^\p c}{\varpi_\fl^n}\bigg)\leq 2\ell^{\lceil\frac{n}{2}\rceil+\min\{v_\fl(b),v_\fl(c)\}},
\end{equation}
where $a^\p$ is the mod $\fl^n$ inverse of $a$.
\end{lem}

\begin{proof}
Let us denote $\Gamma_n=1+\fl^n\O_{F,\fl}$, $Q(j)=a_0b j+ a_0^\p c j^\p$, $h=\lceil\frac{n}{2}\rceil$, and $w=\min\{v_\fl(b),v_\fl(c)\}$. 
Let us set $d\in\O_{F,\fl}$ and $r\in \Gamma_v$. As the map $a\mapsto\operatorname{\textbf{e}}_F\big(\frac{a}{\varpi_\fl^n}\big)$ is an additive character on $\O_{F,\fl}$, we have
\begin{align*}
Q\big(r(1+\varpi_\fl^h d)\big)
&\equiv a_0 b r+a_0 b r\varpi_\fl^h d+a_0^\p c r^\p\big(1-\varpi_\fl^h d+O(\varpi_\fl^{2h})\big)\\
&\equiv Q(r)+\varpi_\fl^{h}d(a_0 br-a_0^\p cr^\p)\ (\text{mod } \fl^n).
\end{align*}
Note that we have the following group isomorphism:
\begin{align*}\begin{split}
r(1+\varpi_\fl^h d)\in\Gamma_v/\Gamma_n\cong\Gamma_v/\Gamma_h\times\Gamma_h/\Gamma_n\cong\Gamma_v/\Gamma_h\times\O_{F,\fl}/\fl^{n-h}\O_{F,\fl}\ni (r,d).
\end{split}\end{align*}
From the above equations, we can rewrite (\ref{exp:sum:esti}) as
\begin{align}\begin{split}
\sum_{r\in\Gamma_v/\Gamma_h}\textbf{e}_F\bigg(\frac{Q(r)}{\varpi_\fl^n}\bigg)
\sum_{d\ \text{mod } \fl^{n-h}}\textbf{e}_F\bigg(\frac{( a_0 br- a_0^\p c r^\p)d}{\varpi_\fl^{n-h}}\bigg) .
\end{split}\end{align}
By the orthogonality of characters of finite group,
the above equation becomes

\begin{align}\begin{split}\label{exp:sum:esti:3}
\ell^{n-h}\sum_{a_0 br\equiv  a_0^\p cr^\p \ (\text{mod }\fl^{n-h})}\textbf{e}_F\bigg(\frac{Q(r)}{\varpi_\fl^n}\bigg),
\end{split}\end{align}
where $r$ runs over $\Gamma_v/\Gamma_h$.
As $v\leq n-h-w$, we can observe that
$$
\{r\in\Gamma_v/\Gamma_h:\  a_0 br\equiv a_0^\p cr^\p\ (\text{mod }\fl^{n-h})\}\leq \frac{2|\Gamma_1/\Gamma_h|}{\ell^{n-h-w}}= 2\ell^{2h-n+w}.
$$
Hence the absolute value of (\ref{exp:sum:esti:3}) is bounded by $2\ell^{h+w}.$
So we are done.
\end{proof}

We have the following lemma about the Fourier-Whittaker coefficients of the cusp form $W_\fN(f\otimes\psi)$:

\begin{lem}\label{invol:proj:coeff:relation} For $f\in S_2(\fN,\chi)$, we have $a_{W_{\fN}(f\otimes\psi)}(\xi\fa_i)=\psi(a_j)a_{W_{\fN}f}(\xi\fa_i)$, where $j$ is the integer such that $\fa_j=\fa_i^{-1}\fN$ in $\operatorname{Cl}(F)$.
\end{lem}
\begin{proof} Note that $f\otimes\psi\in S_2(\fN,\chi\psi^2)$. By the definition of $W_\fN f$ and $f\otimes\psi$, we have
$$
W_{\fN}(f\otimes\psi)(g)=\psi^{-1}\big(\det(g)\big)\psi(\varpi_\fN)W_\fN f(g).
$$
Hence we obtain $a_{W_{\fN}(f\otimes\psi)}(\xi\fa_i)=\psi(a_i^{-1}\varpi_{\fN})a_{W_{\fN}f}(\xi\fa_i)$ for $\xi\in F^\times$ by the Proposition \ref{four:whit:exp:hida} and the equation (\ref{integral:rep}).
So we conclude the proof.
\end{proof}

Let us recall that $H_F$ is the group of Hilbert characters of $F$.
For $f\in S_2(\fN,\chi)$, we obtain
\begin{align}\label{special:Lvalue:pairing:repn:2}
\frac{1}{h_F}\sum_{\psi\in H_F}\psi(a_i)^{-1}\Gamma_{F,2}(1)L\Big(1,f\otimes\psi,\frac{a}{\varpi_\fl^n}\Big)=\xi^{(i)}_{n,\fl}(a)\cap\delta^{(i)}(f_{(i)})
\end{align}
by using the equations (\ref{special:value:of:twisted:modular:L:functions}), (\ref{integral:evaluation}) and the orthogonality of characters.

Note that $S_2\big(U_1(\fN)\big)$ is a finite dimensional complex vector space. From the pairing (\ref{lefschetz:poincare:pairing}) and the generalized Eichler-Shimura isomorphism, we induce the following non-degenerate pairing:
\begin{align*}
\langle\cdot,\cdot\rangle:H_1(X_1(\fN),\C)\times S_2\big(U_1(\fN)\big)\rightarrow\C,\ \langle c,f \rangle\mapsto\sum_{i=1}^{h_F}c_i\cap\delta^{(i)}(f_{(i)}),
\end{align*}
where $c_i$ is the $i$-th component of a homology class $c$ under the isomorphism $H_1(X_1(\fN),\C)\cong\bigoplus_{i=1}^{h_F} H_1(X_1^i(\fN),\C)$. 
 
For an open subset $U_v(a_0):=a_0+\fl^v\O_{F,\fl}$ of $\O_{F,\fl}^\times$ and an element $b\in\O_{F,\fl}$, let us define a special modular symbol $\Upsilon_{n,\fl}^{(i)}\big(U_v(a_0),b\big)$ by
$$
\Upsilon_{n,\fl}^{(i)}\big(U_v(a_0),b\big):=\frac{1}{\ell^{n-v}}\sum_{a\equiv a_0 (\fl^v)}\mathbf{e}_F\Big(\frac{-ab}{\varpi_\fl^n}\Big)\cdot\xi^{(i)}_{n,\fl}(a),
$$
where $a-a_0$ runs over $\fl^v\O_{F,\fl}/\fl^n\O_{F,\fl}$. 
By the generalized Manin-Drinfeld theorem, we have $\Upsilon_{n,\fl}^{(i)}\big(U_v(a_0),b\big)\in H_1(X_1^i(\fN),\C)$. If there is no ambiguity on $\fl$, we just denote $\Upsilon_{n,\fl}^{(i)}$ by $\Upsilon_{n}^{(i)}$. Then we have the following estimation: 

\begin{prop}\label{pairing:esti} Let $\beta$ be a non-zero element of $\fa_i^{-1}$ and $f\in S_2\big(U_1(\fN)\big)$. For $n>\max\{2v+2v_\fl(\beta),v+\operatorname{log}_\ell\big(4N(\beta)\big)\}$ and $\epsilon>0$, we have
\begin{align*}\begin{split}
\langle\Upsilon_{n,\fl}^{(i)}(U_{v}(a_0),\beta),f\rangle=\frac{a_{f}(\beta\fa_i)}{N(\beta\fa_i)}\big(1+O(\ell^{-\frac{n}{6}})\big)+O\Big(\frac{\ell^v}{\ell^{(\frac{3}{16}-\e)n}}\Big)+O\bigg(\frac{\ell^{v+v_\fl(\beta)+1}}{\ell^{(\frac{3}{32}-\e)n}}\bigg),
\end{split}\end{align*}
where the implicit constant does not depend on $\beta$, $n$, and $\ell$.
\end{prop}

\begin{proof}
Let $f_\chi$ be the $\chi$-component of $f$ with respect to the isomorphism $S_2\big(U_1(\fN)\big)\cong\bigoplus_\chi S_2(\fN,\chi)$. Then we have $\langle -,f \rangle=\sum_{\chi}\langle -,f_\chi \rangle$. Note that the size of the sum $\sum_\chi$ depends only on $h_F$ and $N(\fN)$. So we can assume that $f\in S_2(\fN,\chi)$. 

For the simplicity, let us denote $U=U_v(a_0)$. By the definition of the special modular symbols and the equation (\ref{special:Lvalue:pairing:repn:2}), we obtain
\begin{align}\begin{split}\label{induced:pairing:eval}
\langle\Upsilon_{n,\fl}^{(i)}\big(U,\beta\big),f\rangle=\frac{1}{\ell^{n-v}}\sum_{a\equiv a_0 (\fl^n)}\mathbf{e}_F\Big(\frac{-a\beta}{\varpi_\fl^n}\Big) \sum_{\psi\in H_F} \frac{\Gamma_{F,2}(1)}{h_F\psi(a_i)}L\Big(1,f\otimes\psi,\frac{a}{\varpi_\fl^n}\Big).
\end{split}\end{align}
By the equation (\ref{approx:func:eq}) for $\fq=\fl^n$, Lemma \ref{invol:proj:coeff:relation}, and the orthogonality of characters, we have the following equation of the $L$-function:
\begin{align}\begin{split}\label{approx:func:eq:2}
\sum_{\psi\in H_F}\frac{\Gamma_{F,2}(1)}{h_F\psi(a_i)}L\Big(1,f\otimes\psi,\frac{a}{\varpi_\fl^n}\Big)=\sum_{0\neq\a\in\fa_i^{-1}}\frac{a_{f}(\a\fa_i)\textbf{e}_F\big(\frac{\a a }{\varpi_\fl^n}\big)}{N(\a\fa_i)|\O_F^\times|}V_{1,1}\bigg(\frac{N(\a\fa_i)}{y}\bigg)& \\
+\chi(\varpi_\fl)^n\sum_{0\neq\a\in\fa_j^{-1}}\frac{a_{W_\fN f}(\a\fa_j)\textbf{e}_F\big(\frac{\a a^\p}{\varpi_\fl^n}\big)}{N(\a\fa_j)|\O_F^\times|}V_{2,1}\bigg(\frac{N(\a\fa_j)y}{N(\fN)\ell^{2n}}\bigg)&,
\end{split}\end{align}
where $a^\p\in\O_{F,\fl}$ is the mod $\fl^{n}$ inverse of $-a$, and $j$ is the integer such that $\fa_j=\fa_i^{-1}\fN$ in $\text{Cl}(F)$. By using the equation (\ref{approx:func:eq:2}), we split the equation (\ref{induced:pairing:eval}) into the two terms:
\begin{align}
\frac{1}{\ell^{n-v}}&\sum_{a\equiv a_0 (\fl^v)}\sum_{0\neq\a\in\fa_i^{-1}}\frac{a_{f}(\a\fa_i)\textbf{e}_F\big(\frac{a(\a -\beta)}{\varpi_\fl^n}\big)}{N(\a\fa_i)|\O_F^\times|}V_{1,1}\bigg(\frac{N(\a\fa_i)}{y}\bigg)  \label{approx:func:eq:2:1st} \\ 
+\frac{\chi(\varpi_\fl)^n}{\ell^{n-v}}&\sum_{a\equiv a_0 (\fl^v)}\sum_{0\neq\a\in\fa_j^{-1}}\frac{a_{W_\fN f}(\a\fa_j)\textbf{e}_F\big(\frac{a^\p \a -a\beta }{\varpi_\fl^n }\big)}{N(\a\fa_j)|\O_F^\times|}V_{2,1}\bigg(\frac{N(\a\fa_j)y}{N(\fN)\ell^{2n}}\bigg). \label{approx:func:eq:2:2nd}
\end{align} 
Note that by the orthogonality of characters, we observe that
\begin{displaymath}
\frac{1}{\ell^{n-v}}\sum_{a\equiv a_0 (\fl^v)}\textbf{e}_F\Big(\frac{ar}{\varpi_\fl^n}\Big) = \left\{ \begin{array}{ll}
\textbf{e}_F\big(\frac{a_0r}{\varpi_\fl^n}\big) & \textrm{if $v_\fl(r)\geq n-v$,}\\
0 & \textrm{otherwise}.
\end{array} \right.
\end{displaymath}
By Lemma \ref{suff:large:n:identity}, we can write (\ref{approx:func:eq:2:1st}) as
\begin{align} 
\frac{a_{f}(\beta\fa_i)}{N(\beta\fa_i)|\O_F^\times|}V_{1,1}\bigg(\frac{N(\beta\fa_i)}{y}\bigg)+\sum_{\a\in X^{(i)}_{\beta,n}}\frac{a_{f}(\a\fa_i)\textbf{e}_F\big(\frac{a_0(\a-\beta)}{\varpi_\fl^n}\big)}{N(\a\fa_i)|\O_F^\times|}V_{1,1}\bigg(\frac{N(\a\fa_i)}{y}\bigg) \label{approx:func:eq:2:1st:char}
\end{align}
where $X^{(i)}_{\beta,n}$ is the set defined by $\{\a\in (\beta+\fl^{n-v})\cap\fa_i^{-1}:N(\a)>N(\beta)\}$. Note that if $\alpha\in X^{(i)}_{\beta,n}(y)$, then we have $4^{-1}\ell^{n-v}<(\ell^{\frac{n-v}{2}}-N(\beta)^{1/2})^2\leq N(\alpha)$ by the triangle inequality,  Let $0\leq\theta<\frac{1}{2}$ be a number such that $|a_f(\fa)|\ll_\e N(\fa)^{\frac{1}{2}+\theta+\epsilon}$ for any ideals $\fa$ of $F$. Then by using the equation (\ref{aux.func.esti}), the estimation of the equation (\ref{approx:func:eq:2:1st:char}) is given by
\begin{equation*}
\frac{a_{f}(\beta\fa_i)}{N(\beta\fa_i)}\bigg(1+O\bigg(\frac{\ell^{\frac{n-v}{2}}}{y^{1/2}}\bigg)\bigg)
\end{equation*}
for $y\geq N(\fa_i)\ell^{n-v}$, where the implicit constants depends only on $D_F$ and $N(\fa_i)$. 
We split the second term of (\ref{approx:func:eq:2:1st:char}) into the two parts:
\begin{equation*}
\text{2nd term of }(\ref{approx:func:eq:2:1st:char})=(*)+(**),\text{ where }(*)=\sum_{N(\a\fa_i)\leq y},\ (**)=\sum_{N(\a\fa_i)>y}.
\end{equation*}
For $a\in\O_{F,\fl}$, define functions $u_{a,n}^{(i)}$ and $U_{a,n}^{(i)}$ by 
\begin{displaymath}
u_{a,n}^{(i)}(\a) = \left\{ \begin{array}{ll}
1 & \textrm{if $\a\in (a+\fl^{n-v})\cap\fa_i^{-1}$},\\
0 & \textrm{otherwise},
\end{array} \right.
\end{displaymath} 
and
$
U_{a,n}^{(i)}(y):=\sum_{N(\a)\leq y}u_{a,n}^{(i)}(\a)=\#\{\a\in (a+\fl^{n-v})\cap\fa_i^{-1}:\ N(\a)\leq y \}.
$ 
By using the Abel summation formula, the bound on $a_f(\fa)$, the equation (\ref{aux.func.esti}), Lemma \ref{lattice:size:esti} and Lemma \ref{suff:large:n:identity}, the estimation of $(*)$ is given by 
\begin{align*}
(*)\ll_{D_F,N(\fa_i),\e} &\sum_{4^{-1}\ell^{n-v}<N(\a)\leq y N(\fa_i^{-1})}\frac{u_{\beta,n}^{(i)}(\a)}{N(\a)^{\frac{1}{2}-\theta-\epsilon}} \\
\ll_{N(\fa_i),\e,\theta} &\frac{U_{\beta,n}^{(i)}(yN(\fa_i^{-1}))}{y^{\frac{1}{2}-\theta-\epsilon}}+\frac{U_{\beta,n}^{(i)}(4^{-1}\ell^{n-v})}{(4^{-1}\ell^{n-v})^{\frac{1}{2}-\theta-\epsilon}} \\
+&\int_{4^{-1}\ell^{n-v}}^{yN(\fa_i^{-1})} \frac{U_{\beta,n}^{(i)}(t)}{t^{\frac{3}{2}-\theta-\epsilon}}dt \ll_{F,N(\fa_i),\e,\theta} \frac{y^{\frac{1}{2}+\theta+\epsilon}}{\ell^{n-v}}+\ell^{(\theta+\epsilon-\frac{1}{2})(n-v)}
\end{align*} 
for $y\geq N(\fa_i)\ell^{n-v}$.
Similarly as above, the estimation of $(**)$ is given by
\begin{align*}
(**)\ll_{F,N(\fa_i),\e} &\sum_{N(\a\fa_i)>y}\frac{u_{\beta,n}^{(i)}(\a)y}{N(\a)^{\frac{3}{2}-\theta-\epsilon}} \ll_{N(\fa_i),\e,\theta}\frac{U_{\beta,n}^{(i)}(yN(\fa_i)^{-1})y}{y^{\frac{3}{2}-\theta-\epsilon}}\\
+&\int^\infty_{yN(\fa_i)^{-1}} \frac{U_{\beta,n}^{(i)}(t)y}{t^{\frac{5}{2}-\theta-\epsilon}}dt \ll_{F,N(\fa_i),\e,\theta} \frac{y^{\frac{1}{2}+\theta+\epsilon}}{\ell^{n-v}}
\end{align*} 
for $y\geq N(\fa_i)\ell^{n-v}$.
In sum, the estimation of the equation (\ref{approx:func:eq:2:1st:char}) is given by
\begin{align*}
\frac{a_{f_\nu}(\beta\fa_i)}{N(\beta\fa_i)}\bigg(1+O\bigg(\frac{\ell^{\frac{n-v}{2}}}{y^{1/2}}\bigg)\bigg)+O\bigg(\frac{y^{\frac{1}{2}+\theta+\epsilon}}{\ell^{n-v}}\bigg)+O(\ell^{(\theta+\epsilon-\frac{1}{2})(n-v)})
\end{align*}
for $y\geq N(\fa_i)\ell^{n-v}$, where the implicit constants depends only on $F,N(\fa_i),\e,\theta$.

On the other hand, we can write the term (\ref{approx:func:eq:2:2nd}) by
\begin{equation}\label{approx:func:eq:2:2nd:char}
\chi(\varpi_\fl)^n\sum_{0\neq\a\in\fa_j^{-1}}\frac{a_{W_\fN f}(\a\fa_j)\text{Kl}_n(\a)}{N(\a\fa_j)|\O_F^\times|}V_{2,1}\bigg(\frac{N(\a\fa_j)y}{N(\fN)\ell^{2n}}\bigg),
\end{equation}
where $\text{Kl}_n(\a)$ is a number defined by
$$
\text{Kl}_n(\a):=\frac{1}{\ell^{n-v}}\sum_{a\equiv a_0 (\fl^v)}\textbf{e}_F\bigg(\frac{-a\beta + a^\p \a }{\varpi_\fl^n}\bigg).
$$
We split the term (\ref{approx:func:eq:2:2nd:char}) into the two parts:
\begin{equation*}
(\ref{approx:func:eq:2:2nd:char})=I+II,\text{ where }I=\sum_{N(\a\fa_j)\leq N(\fN)\ell^{2n}/y},\ II=\sum_{N(\a\fa_j)>N(\fN)\ell^{2n}/y}.
\end{equation*}
Since $n>2v+2v_\fl(\beta)$, we have $\text{Kl}_n(\a) \leq \ell^{-\frac{n}{2}+v+v_\fl(\beta)+1}$ by Lemma \ref{exp:sum}.
Note that we have $a_{W_\fN f}(\fa)\ll_\e N(\fa)^{\frac{1}{2}+\theta+\e}$. Therefore, by the equation (\ref{aux.func.esti}), the estimations of $I$ and $II$ are given by
\begin{align*}
I &\ll_{D_F,N(\fN),\e}\frac{\ell^{v+v_\fl(\beta)+1}}{\ell^{\frac{n}{2}}}\sum_{N(\a)\leq N(\fN\fa_j^{-1})\ell^{2n}/y}\frac{1}{N(\a)^{\frac{1}{2}-\theta-\e}} \\
 &\ll_{N(\fa_j),N(\fN),\theta,\e} \frac{\ell^{(\frac{1}{2}+2\theta+2\e)n+v+v_\fl(\beta)+1}}{y^{\frac{1}{2}+\theta+\e}}, \\
II &\ll_{D_F,N(\fN),\e} \ell^{v+v_\fl(\beta)+1+\frac{3n}{2}}\sum_{N(\a)> N(\fN\fa_j^{-1})\ell^{2n}/y}\frac{1}{N(\a)^{\frac{3}{2}-\theta-\e}y} \\
 &\ll_{N(\fa_j),N(\fN),\theta,\e} \frac{\ell^{(\frac{1}{2}+2\theta+2\e)n+v +v_\fl(\beta)+1}}{y^{\frac{1}{2}+\theta+\e}}.
\end{align*}
Hence we have
\begin{align*}
(\ref{approx:func:eq:2:2nd:char})=I+II=O\bigg(\frac{\ell^{(\frac{1}{2}+2\theta+2\e)n+v+v_\fl(\beta)+1}}{y^{\frac{1}{2}+\theta+\e}}\bigg),
\end{align*} where the implicit constant depends only on $D_F,N(\fa_j),N(\fN),\e$ and $\theta$. 

In sum, we have
\begin{align*}\begin{split}
\langle\Upsilon_{n,\fl}^{(i)}(U_v(a_0),\beta),f\rangle=&\frac{a_{f_\nu}(\beta\fa_i)}{N(\beta\fa_i)}\bigg(1+O\bigg(\frac{\ell^{\frac{n-v}{2}}}{y^{1/2}}\bigg)\bigg)+O\bigg(\frac{y^{\frac{1}{2}+\theta+\epsilon}}{\ell^{n-v}}\bigg) \\
&+O(\ell^{(\theta+\epsilon-\frac{1}{2})(n-v)})+O\bigg(\frac{\ell^{(\frac{1}{2}+2\theta+2\e)n+v+v_\fl(\beta)+1}}{y^{\frac{1}{2}+\theta+\e}}\bigg)
\end{split}\end{align*}
for $y\geq N(\fa_i)\ell^{n-v}$, where the implicit constant depends only on $F$, $N(\fa_i)$, $N(\fa_j)$, $N(\fN)$, $\e$ and $\theta$. Note that $\ell$ is coprime to $\fa_i^{-1}\fN$ for each $i$, so the implicit constant does not depend on $\ell$. By Nakasuji \cite[Corollary 1.1]{Nakasuji}, we can choose $\theta=\frac{7}{64}$. By letting $y=\ell^{\frac{4}{3}(n-v)}$, we conclude the proof.
\end{proof} 


\subsection{Full-rankness results}

From the estimation in the previous subsection, we can prove the full-rankness of $M_{n,\fl}$ in $H_1(X_1(\fN),\Z)$. We need the following fundamental lemma:

\begin{lem}\label{fund:lin:alg} Let $X$ and $Y$ be finite dimensional complex vector spaces with a non-degenerate pairing $\langle\cdot,\cdot\rangle:X\times Y\rightarrow\C$. 
For any sequence of proper subspaces $\{X_m\}_m$ of $X$, there is a non-zero vector $y\in Y$ and a sequence $\{m_j\}$ in $\Z_{>0}$ such that $\lim_{j\rightarrow\infty}\langle u_{m_j},y \rangle=0$ for any sequence $\{u_m\}_{m>0}$ of vectors $u_m\in X_m$ of bounded norm.
\end{lem}

\begin{proof} Let us choose $m>0$. Since $X_m$ is proper in $X$, there is a non-zero element $y_m\in Y$ such that $|y_m|_Y=1$ and $\langle x,y_m \rangle=0$ for all $x\in X_m$. Since the unit ball of $Y$ is compact, there is a sequence $\{m_j\}_j$ in $\Z_{>0}$ and a vector $y$ in the unit ball of $Y$ such that $y_{m_j}$ converges to $y$. Then by applying the Cauchy-Schwarz inequality, we have
$$
|\langle u_{m_j},y\rangle|=|\langle u_{m_j},y-y_{m_j}\rangle| \leq |u_{m_j}|\cdot|y-y_{m_j}|
$$
for any bounded sequence $\{u_{m}\}_m$ in $X$. By taking the limit $j\rightarrow\infty$, we finish the proof.
\end{proof}

\begin{thm}\label{vertical:direction}
For $n\geq 22$, the module $M_{n,\fl}$ is of full-rank in $H_1(X_1(\fN),\Z)$ for the prime ideals $\fl$ of sufficiently large norm. 
\end{thm}
\begin{proof} 
Let us set $n\geq 22$.
We need to show that $M_{n,\fl}\otimes_\Z\C = H_1(X_1(\fN),\C)$ for the prime ideals $\fl$ of sufficienly large norm.  
Let us assume the contrary, i,e., assume that there is a sequence $\{\fl_m\}_m$ of prime ideals such that $M_{n,\fl_m}\otimes_\Z\C$ is proper in $H_1(X_1(\fN),\C)$ and $N(\fl_m)\rightarrow\infty$ as $m\rightarrow\infty$. Note that without loss of generality, we can discard the prime ideals whose norm divides $\fa_i^{-1}\fN$ for some $i$ from the sequence $\{\fl_{m}\}_m$ . 
By Lemma \ref{fund:lin:alg}, there is a non-zero element $f\in S_2\big(U_1(\fN)\big)$ of bounded norm and a sequence $\{m_j\}$ in $\Z_{>0}$ such that for any bounded sequence $\{u_m\}$ in $M_{n,\fl_m}\otimes_\Z\C$, we have 
\begin{equation}\label{pairing:value:zero}
\lim_{j\rightarrow\infty}\langle u_{m_j},f\rangle=0.
\end{equation}
Since $f$ is non-zero, there is a non-zero integral ideal $\fa=\beta\fa_i$ such that $a_{f}(\beta\fa_i)\neq 0$. 

Let us consider a sequence $\{\Upsilon_{n}^{(i)}(U_{\fl_{m}},\beta)\}_{m}$, where $U_\fl=1+\fl\O^\times_{F,\fl}$. Note that $\Upsilon_{n}^{(i)}(U_{\fl_{m}},\beta)\in M_{n,\fl_m}\otimes_\Z\C$ and $v_{\fl_m}(\beta)=0$ for sufficiently large $m$. 
From Proposition \ref{pairing:esti}, we obtain
\begin{equation}\label{special:symbol:pairing:eval}
\langle\Upsilon_{n}^{(i)}(U_{\fl_m},\beta),g\rangle=\frac{a_g(\beta\fa_i)}{N(\beta\fa_i)}+o(1)
\end{equation}
as $m\rightarrow\infty$ for any $g\in S_2\big(U_1(\fN)\big)$, where the implicit constant does not depend on $\{\fl_m\}$.
So the sequence $\{\Upsilon_{n}^{(i)}(U_{\fl_{m}},\beta)\}_{m}$ is bounded. By Proposition \ref{pairing:esti} and the equation (\ref{pairing:value:zero}), we obtain
$$
\lim_{m\rightarrow\infty}\langle\Upsilon_{n}^{(i)}(U_{\fl_m},\beta),f\rangle=\frac{a_{f}(\beta\fa_i)}{N(\beta\fa_i)}=0,
$$
which is a contradiction.
\end{proof}

From the above result, we obtain the full-rankness theorem of horizontal direction, which has a key role to prove the result toward Greenberg's conjecure of Bianchi modular version:

\begin{thm}[Horizontal direction]\label{full:rank:horizontal} For an integer $n\geq 22$, there is a prime ideal $\fp_{n}=\a_{n}\O_F$ whose norm is coprime to $\fN\prod_{i=1}^{h_F}\fa_i^{-1}$ such that $M_{n,\fp_{n}}$ is of full-rank in $H_1(X_1(\fN),\Z)$. And there is an arithmetic progression $\mathfrak{X}_n\subset\big\{\a\O_F:\a\in \a_{n}+\fN\prod_{i=1}^{h_F}\fa_i^{-1}\big\}$ of ideals such that $M_{n,\fp_{n}}\subset M_{n,\fp}$ for the prime ideals $\fp$ in $\mathfrak{X}_n$. 
\end{thm}

\begin{proof} Let us choose an integer $n\geq 22$. By Theorem \ref{vertical:direction}, there is a prime ideal $\fp_{n}=\a_{n}\O_F$ of sufficiently large norm so that $N(\fp_{n})$ is coprime to $\fN\prod_{i=1}^{h_F}\fa_i^{-1}$ and $M_{n,\fp_{n}}$ is of full-rank in $H_1(X_1(\fN),\Z)$. 
Choose a generator $\xi$ of the module $M_{n,\fp_{n}}$. Note that we have $\xi=\xi^{(i)}_{n,\fp_{n}}(\beta)$ for some $\beta\in(\O_{F}/\fp_{n}^n)^\times$.
From Proposition \ref{adelic:modular:symbol:reprn} and our assumption on $\fp_{n}$, we obtain
$$
\xi=\xi^{(i)}_{n,\fp_n}(\beta)=\{\infty,\beta\a_{n}^{-n}\}_{\ovl{\Gamma^i_1(\fN)}}.
$$ 
Let us set $B=\prod_{\beta\in(\O_F/\fp_{n}^n)^\times}\beta$ and $\mathfrak{C}=\prod_{i=1}^{h_F}\fa_i^{-1}$. Let us define an arithmetic progression $\mathfrak{X}_n$ of ideals by $\mathfrak{X}_n:=\{\a\O_F:\a\in\a_{n}+B\mathfrak{C}\fN,\ N(\fp)\geq N(\fp_n)\}$. Let us choose a prime ideal $\fp=\a\O_F\in\mathfrak{X}_n$. Then we observe that $\fp$ is coprime to $\fa_i^{-1}$ for each $i$ so that the symbols $\xi_{n,\fp}^{(i)}(a)$ and modules $M_{n,\fp}$ are well defined. Also we have $\xi^{(i)}_{n,\fp}(\beta)=\{\infty,\beta\a^{-n}\}_{\ovl{\Gamma^i_1(\fN)}}$ and $\a^{n}=\a_{n}^{n}+BCN$ for some $C\in\mathfrak{C}$ and $N\in\fN$.

Consider the action of $\SL_2(F)$ on $Y_1(\fN)$. If $c\in\fa_i^{-1}\fN$, then $\gamma_c=\begin{spmatrix} 1 & 0 \\ c & 1 \end{spmatrix}\in\GL_2(F)$ and $\gamma_{c,\infty}\in\Gamma_1^i(\fN)$. For $b\in F$ and $y\in\R_{>0}^\times$, we have
$$
t_i(b;y)=\gamma_ct_i(b;y)=t_i\gamma_{c,\infty}(b;y)\begin{pmatrix} 1 & 0 \\ a_i c^{(\infty)} & 1 \end{pmatrix}=t_i\begin{pmatrix} 1 & 0 \\ c & 1 \end{pmatrix}_\infty(b;y)
$$
in $Y_1(\fN)$. Note that $\{\infty,b\}_{\Gamma_1^i(\fN)}$ is represented by the set $\{t_i(b;y):y\in\R_{>0}^\times\}$. Also note that the action of $\SL_2(\C)$ on $Y_1^i(\fN)$ preserve the metric, and 
$$
\begin{pmatrix} 1 & 0 \\ c & 1 \end{pmatrix}_\infty(b;y)\mapsto\bigg(\frac{\ovl{c}y^2+b\ovl{(bc+1)}}{|c|^2y^2+|bc+1|^2},\frac{y}{|c|^2y^2+|bc+1|^2}\bigg)
$$ 
under the bijection $\SL_2(\C)/\SU_2(\C)\rightarrow\mathscr{H}_3$.
Thus we have 
$$
\{\infty,b\}_{\ovl{\Gamma^i_1(\fN)}}=\Big\{\frac{1}{c},\frac{b}{bc+1}\Big\}_{\ovl{\Gamma^i_1(\fN)}}=\gamma_c\{\infty,0\}_{\ovl{\Gamma^i_1(\fN)}}+\{0,\infty\}_{\ovl{\Gamma^i_1(\fN)}}+\Big\{\infty,\frac{b}{bc+1}\Big\}_{\ovl{\Gamma^i_1(\fN)}}.
$$

Let us set $b=\beta\a_{n}^{-n}\in F^\times$ and $c=\beta^{-1}BCN\in\fN\prod_{i=1}\fa_i^{-1}$. Note that $\gamma_c\in\GL_2(F)$ and $\gamma_{c,\infty}\in\bigcap_{i=1}^{h_F}\Gamma_1^i(\fN)$. Then by the above equations, we obtain
$$
\xi=\xi_{n,\fp_{n}}^{(i)}(\beta)=\{\infty,b\}_{\ovl{\Gamma^i_1(\fN)}}=\{\infty,\beta\a^{-n}\}_{\ovl{\Gamma^i_1(\fN)}}=\xi_{n,\fp}^{(i)}(\beta),
$$
which is a generator of $M_{n,\fp}$.
Thus $M_{n,\fp_{n}}$ is a submodule of $M_{n,\fp}$, which finishes the proof.
\end{proof}

\begin{rem}\label{horizontal:module:indices}
Let us follow the notations in Theorem \ref{full:rank:horizontal}.
Let us define a positive integer as follows:
$$
\mathfrak{h}_n:=[H_1(X_1(\fN),\Z):M_{n,\fp_{n}}].
$$ 
Then by Theorem \ref{full:rank:horizontal}, the indices of $M_{n,\fp}$ in $H_1(X_1(\fN),\Z)$ divide $\mathfrak{h}_n$ for the prime ideals $\fp\in\mathfrak{X}_n$. Let $q$ be an odd prime number does not divide $\mathfrak{h}_n$. Then we have
$$
M_{n,\fp}\otimes_\Z\ovl{\mathbb{F}}_q=H_1(X_1(\fN),\ovl{\mathbb{F}}_q)
$$
for the prime ideals $\fp\in\mathfrak{X}_n$ by Theorem \ref{full:rank:horizontal} and the universal coefficient theorem.
\end{rem}

Let $f\in S_2(\fN,\chi)$ be a normalized Hecke eigenform and $\fp$ a prime ideal coprime to $\fa_i^{-1}$ for each $i$. For an open subset $Z\subset\O_{F,\fp}$, define a module 
$M^f_{n,\fp}(Z)$ by
$$
M^f_{n,\fp}(Z):=\langle a_f(\fp)\xi_n(r)-\chi(\varpi_\fp)\xi_{n-1}(r):\ r\in Z\rangle\cap H_1(X_1(\fN),\ovl{\Z}),
$$
which depends on $a_f(\fp)$. We denote $M^f_{\fp,n}(\O_{F,\fp})$ by $M^f_{\fp,n}$.
Then we have a variation of Theorem \ref{full:rank:horizontal}:

\begin{prop}\label{full:rank:horizontal:variation} 
Let us follow the notations in Theorem \ref{full:rank:horizontal} and Remark \ref{horizontal:module:indices}. Let $\fp$ be an $f$-ordinary prime ideal $\fp\in\mathfrak{X}_n$ coprime to $\mathfrak{h}_n$. Let us denote $p$ the odd rational prime lying below $\fp$.
Then we have $M^f_{n,\fp}\otimes_\Z\mathbb{F}_{p}=H_1(X_1(\fN),\ovl{\mathbb{F}}_{p})$.
\end{prop}

\begin{proof} 
For each open subset $Z\subset\O_{F,\fp}^\times$, let us denote $M_{n,\fp}^{(i)}(Z):=\langle \xi_{n,\fp}^{(i)}(r): r\in Z\rangle\cap H_1(X_1^i(\fN),\Z)$ and $M^{f,(i)}_{n,\fp}(Z):=M^{f}_{n,\fp}(Z)\cap H_1(X_1^i(\fN),\ovl{\Z})$. For $a\in (\O_F/\fp^{n-1})^\times$, let us choose an element $\xi\in M^{(i)}_{n,\fp}(a+\fp^{n-1}\O_{F,\fp})$. Then $\xi=\sum_{b\in\fp^{n-1}/\fp^n}N_b\xi^{(i)}_{n,\fp}(a+b)$, where $N_b$ is a non-negative integer. Then by the proof of Theorem \ref{full:rank:horizontal}, we have
$$
\xi^{(i)}_{n,\fp}(a+b)=\Big\{\infty,\frac{a+b}{\a^{n}}\Big\}_{\ovl{\Gamma^i_1(\fN)}},
$$
where $\a$ is a generator of an ideal $\fp$ of $F$. 
Taking the boundary map $\partial$ on the symbol $\xi$, we obtain 
$$
\sum_b N_b\Big\{\frac{a+b}{\a^{n}}\Big\}_{\ovl{\Gamma^i_1(\fN)}}=\sum_b N_b\{\infty\}_{\ovl{\Gamma^i_1(\fN)}}.
$$
Note that the above cusps are not equivalent by the definition of $\Gamma_1^i(\fN)$ and $\mathfrak{X}_n$. Thus we have $\sum_b N_b=0$. From this relation, we obtain
$$
a_f(\fp)\xi=a_f(\fp)\sum_{b}N_r\xi^{(i)}_{n,\fp}(a+b)=\sum_{b}N_b\big(a_f(\fp)\xi^{(i)}_{n,\fp}(a+b)-\chi(\varpi_\fp)\xi^{(i)}_{n-1,\fp}(a)\big),
$$
which implies that 
$$
M^{(i)}_{n,\fp}(a+\fp^{n-1}\O_{F,\fp})\otimes_\Z\ovl{\F}_{p}\subset M^{f,(i)}_{n,\fp}(a+\fp^{n-1}\O_{F,\fp})\otimes_\Z\F_{p}\subset H_1(X^i_1(\fN),\ovl{\F}_{p}).
$$
By using Remark \ref{horizontal:module:indices} and summing the spaces over $a$ and $i$, we conclude the proof.
\end{proof}

\section{Cohomology classes and special $L$-values}\label{conj:eisenstein:special}
In this section, we define integral $L$-values of Hecke eigenforms and study their residual non-vanishing. Also we compare the integral $L$-values and algebraic $L$-values.

\subsection{Special $L$-values}
In this section, we follow Namikawa \cite[Section 2]{Namikawa} to define the cohomology classes of integral coefficients and the integral $L$-values. 

Let $p$ be a odd prime number coprime to $h_FD_F$ and $\fp$ a prime ideal of $F$ lying above $p$. Let $f\in S_k(\fN,\chi)$ be a normalized Hecke eigenform, and $\O=\Z_p[\{a_f(\fl):\fl\}]$. For a cohomology class $\delta\in H_*^1(Y_1(\fN),R)$, let us denote by $\delta_i$ the image of $\delta$ under the map $H_*^1(Y_1(\fN),\O)\rightarrow H_*^1(Y_1^i(\fN),R)$, where $*\in\{\operatorname{c},\text{par}\}$ and $R=\O$, $\ovl{\F}_p$ or $\C$. Let $\delta:S_2\big(U_1(\fN)\big)\rightarrow H^1_{\text{par}}(Y_1(\fN),\C)$ be the generalized Eichler-Shimura isomorphism defined by $\delta(f):=\sum_{i=1}^{h_F}\delta^{(i)}(f_{(i)})$. 

Let us fix an isomorphism $\C\cong\C_p$, then we have the scalar extension map $H^1(Y_1(\fN),\O)\rightarrow H^1(Y_1(\fN),\C)$.
Let us recall that there is a section $\iota$ of the map $H^1_{\operatorname{c}}(Y_1(\fN),\C)\rightarrow H^1_{\text{par}}(Y_1(\fN),\C)$. 
Define the \emph{cuspidal cohomology group} of coefficient $\O$ by 
$$
H^1_{\text{cusp}}(Y_1(\fN),\O):=\operatorname{Im}(\iota)\cap H^1_{\operatorname{c}}(Y_1(\fN),\O),
$$
where the intersection is defined via the scalar extension map. Then the $f$-eigenspace $H^1_{\text{cusp}}(Y_1(\fN),\O)[f]$ is a free $\O$-module of rank one. Fix a $\O$-free generator $\eta_{f,{\operatorname{c}}}$ of $H^1_{\text{cusp}}(Y_1(\fN),\O)[f]$, which is determined up to multiplication by $\O^\times$. We define the \emph{Hida's canonical period} $\Omega_{f,{\operatorname{c}}}\in\C^\times$ by the number which satisfies $\iota\big(\delta(f)\big)=\Omega_{f,{\operatorname{c}}}\eta_{f,{\operatorname{c}}}$. 

For a finite order Hecke character $\phi$ whose conductor $\mathfrak{f}(\phi)$ is coprime to $\fa_i^{-1}$, let us define a \emph{twisted modular symbol} $\Lambda^{(i)}(\phi)$ by
$$
\Lambda^{(i)}(\phi):=\sum_{u\in R_{\mathfrak{f}(\phi)}} \phi(a_i)\phi(\varpi_{\mathfrak{f}(\phi)}u)\xi_{u}^{(i)}\in H_1(X_1^i(\fN),C_1^i(\fN),\Z[\phi]) .
$$
\begin{rem}
By Namikawa \cite[Lemma 2.2]{Namikawa}, a relative homology classes $\Lambda^{(1)}(\phi)$ can be lifted to a homology class when $h_F=1$.
\end{rem}

When $\mathfrak{f}(\phi)$ is coprime to $\fa_i^{-1}$ for each $i$, we define the \emph{integral L-value} $\mathcal{L}_f(\phi)$ by
$$
\mathcal{L}_f(\phi):=\sum_{i=1}^{h_F}\Lambda^{(i)}(\phi)\cap\eta_{f,\operatorname{c},i}\in\O[\phi].
$$
Then by the equation (\ref{special:value:of:twisted:modular:L:functions}) and (\ref{integral:evaluation}), we have
$$
\mathcal{L}_f(\phi)=\sum_{i=1}^{h_F}\Lambda^{(i)}(\phi)\cap\frac{\iota\big(\delta^{(i)}(f_{(i)})\big)}{\Omega_{f,\operatorname{c}}}=\frac{|D_F||\O_F^\times| G(\phi)L(1,f\otimes\phi)}{8\pi^2\Omega_{f,{\operatorname{c}}}}\in\mathcal{O}[\phi].
$$

Similar with the case of $H_c^1$, the $f$-eigenspace $H^1_{\text{par}}(Y_1(\fN),\O)^\p[f]$ is free $\O$-module of rank one by Hida \cite[Section 8]{hida1994critical}. Fix a free generator $\eta_{f}$ of $H^1_{\text{par}}(Y_1(\fN),\O)^\p[f]$, which is determined up to multiplication by $\O^\times$. We define the \emph{period} $\Omega_{f}\in\C^\times$ by a complex number such that $\delta(f)=\Omega_{f}\eta_{f}$. Note that we have the \emph{algebraic L-value} $L_f(\phi)$ of $f$ by Hida \cite[Theorem 8.1]{hida1994critical}: 
$$
L_f(\phi):=\frac{|D_F||\O_F^\times| G(\phi)L(1,f\otimes\phi)}{8\pi^2\Omega_{f}}\in\ovl{\Q}.
$$
To compare the $L$-values $\mathcal{L}_f(\phi)$ and $L_f(\phi)$, we have to compare the periods $\Omega_{f,\operatorname{c}}$ and $\Omega_{f}$.
We have the following lemma about the Hecke action on the cohomology classes:
\begin{lem}\label{cohomology:hecke:action} Let $j$ be the integer such that $\fa_i\fp^{-1}=\fa_j$ in $\operatorname{Cl}(F)$. Then we have 
$$
T_\fp\eta_{f,i}=a_f(\fp)\eta_{f,j},\ T_\fp\eta_{f,\operatorname{c},i}=a_f(\fp)\eta_{f,\operatorname{c},j}.
$$
\end{lem}

\begin{proof} 
Let us set $\pi_i(f):=h_F^{-1}\sum_{\psi\in H_F}\psi^{-1}(a_i)f\otimes\psi.$ 
By the definition of $f\otimes\psi$, we have
$$
\delta\big(\pi_i(f)\big)=\frac{1}{h_F}\sum_{\psi\in H_F}\sum_{j=1}^{h_F}\psi^{-1}(a_i)\delta^{(j)}\big((f\otimes\psi)_{(j)}\big)=\delta^{(i)}(f_{(i)}).
$$
Also we have
$$
T_\fp\pi_i(f)=\frac{1}{h_F}\sum_{\psi\in H_F}a_f(\fp)\psi^{-1}(\varpi_\fp^{-1} a_i)f\otimes\psi=a_f(\fp)\pi_j(f).
$$
From the above equations, we have
$$
T_\fp\eta_{f,\operatorname{c},i}=T_\fp\frac{\iota\big(\delta^{(i)}(f_{(i)})\big)}{\Omega_{f,\operatorname{c}}}=\frac{\iota\big(T_\fp\delta(\pi_i(f))\big)}{\Omega_{f,\operatorname{c}}}=a_f(\fp)\eta_{f,\operatorname{c},j}.
$$
Note that we can do the same process for $\eta_f$, so we are done.
\end{proof}

Let us introduce the \emph{non-Eisenstein condition}:
$$
(\mathbf{Non\text{-}Eis})\ a_f(q\O_F)-N(q)-1\in\O^\times\text{ for some prime element }q\equiv 1\ (\text{mod }p\fN).
$$
Then this condition implies following:
\begin{lem}\label{periods:compare} Assume $(\mathbf{Non\text{-}Eis})$, then $\Omega_{f}/\Omega_{f,{\operatorname{c}}}\in\O^\times$.
\end{lem}
\begin{proof} We follow the proof of Namikawa \cite[Proposition 2.6]{Namikawa}. By our assumption, there is a prime element $q\equiv 1\ (\text{mod }p\fN)$ of $F$ such that $\epsilon_q:=a_f(q\O_F)-N(q)-1\in\O^\times$. Let us denote by $\mathcal{E}_q$ the operator $T_{q\O_F}-N(q)-1$. 

From the sequence (\ref{cohom:mod:symb}), we obtain a Hecke-equivariant exact sequence:
\begin{equation*}
\begin{tikzcd} H^0(\partial Y_1(\fN)^*,\O) \arrow{r}{{\operatorname{c}}} & H^1_{\operatorname{c}}(Y_1(\fN),\O) \arrow{r}{\text{par}} & H^1_{\text{par}}(Y_1(\fN),\O) \arrow{r} & 0. \end{tikzcd}
\end{equation*}

We can consider $\eta_f$ to be the element of $H_{\text{par}}^1(Y_1(\fN),\O)$ or $H_{\text{par}}^1(Y_1(\fN),\C)$. Since the map $\text{par}$ is Hecke-equivariant, $\text{par}(\eta_{f,{\operatorname{c}}})=\a\eta_f$ for some $\a\in\O$ by the definition of $\eta_f$. Choose a class $\eta\in H^1_{\operatorname{c}}(Y_1(\fN),\O)$ such that $\text{par}(\eta)=\eta_f$. 
Then we have $\mathcal{E}_q(\eta_{f,{\operatorname{c}},1}-\a\eta_1)=0$ by Namikawa \cite[Lemma 2.3]{Namikawa}. Since $\eta_{f,{\operatorname{c}}}$ is an element of the eigenspace, we obtain $\a\mathcal{E}_q\eta_1=\epsilon_q\eta_{f,{\operatorname{c}},1}$ by Lemma \ref{cohomology:hecke:action}, which is a $\O$-genetator of $H_{\operatorname{cusp}}^1(Y_1^1(\fN),\O)[f]$. 
As $\mathcal{E}_q\eta_1\in H_{\operatorname{c}}^1(Y_1^1(\fN),\O)$, we have $\mathcal{E}_q\eta_1=\beta\eta_{f,{\operatorname{c}},1}$ for some $\beta\in\O$. Taking the map par to the above equation, we obtain $\epsilon_q\eta_{f,1}=\text{par}(\mathcal{E}_q\eta_1)=\a\beta\eta_{f,1}$. Note that $\eta_{f,1}$ is non-trivial since $f$ is an eigenform. Thus $\a$ is a $p$-adic unit and 
$$
\frac{\Omega_f}{\Omega_{f,{\operatorname{c}}}}\eta_f=\text{par}(\eta_{f,{\operatorname{c}}})=\a\eta_f.
$$ 
Therefore, the ratio $\Omega_{f}/\Omega_{f,{\operatorname{c}}}$ is a $p$-adic unit.
\end{proof}

From Lemma \ref{periods:compare}, the ratio of $\mathcal{L}_f(\phi)$ and $L_f(\phi)$ is a $p$-adic unit under the assumption (\textbf{Non-Eis}). Also we can compare the pairings (\ref{lefschetz:poincare:pairing}) and (\ref{lefschetz:poincare:pairing:2}) for $R=\O$.

\subsection{Residual non-vanishing}
In this subsection, we add assumption on $p$ that $F$ does not contain the primitive $p$-th power roots and $p$ is coprime to $|\O_F^\times|\fN$. Let $\mathfrak{M}_p$ be the conductor of the extension $F(\zeta_p)/F$ and $\fM:=\fM_p\fN$. 
Note that by the straightforward computation, we have
$$
\ovl{\Gamma_1^i(\fN)}=\bigg\{\begin{pmatrix} a & b \\ c & d \end{pmatrix}\in\SL_2(\O_F):a,d\in 1+\fN,\ b\in\fa_i,\ c\in\fN\fa_i^{-1}\ \bigg\}.
$$

Let us denote $S(\mathfrak{M}_p,\fN)$ the set of pairs $(b,d)$ of elements $b,d\in F$ such that $b$ is coprime to $d\mathfrak{M}_p$ and $d\equiv 1\ (\text{mod }\mathfrak{M})$. From Namikawa \cite[Lemma 3.2]{Namikawa}, for any $(b,d)\in S(\mathfrak{M}_p,\fN)$, there is an infinite set $P_{b,d}$ of prime elements of $F$ such that for each $\pi\in P_{b,d}$, we have
\begin{enumerate}
\item $\big\{0,\frac{b}{d}\big\}_{\ovl{\Gamma_1^1(\fN)}}=\big\{0,\frac{b}{\pi}\big\}_{\ovl{\Gamma_1^1(\fN)}}$.
\item $\pi\in 1+\fN$.
\item $N(\pi)-1$ is coprime to $p$.
\item $N(\pi)-1\neq |\O_F^\times|$.
\end{enumerate}
Let us denote by $\mathfrak{Y}$ the set of finite order Hecke characters over $F$ such that if $\psi\in\mathfrak{Y}$, then the conductor of $\psi$ is $\pi\O_F$ for some $\pi\in\bigcup_{(b,d)\in S(\mathfrak{M}_p,\fN)}P_{b,d}$. Note that we have the following conjecture on the generalization of the result of Stevens \cite{stevens1958cuspidal}:
\begin{conj}\label{homology:generate:padic:unit:general}
Assume $(\mathbf{Non\text{-}Eis})$. Then there exist infinitley many Hecke characters $\phi$ such that $\mathcal{L}_f(\phi)$ is a $p$-adic unit.
\end{conj}
By using the result of Namikawa \cite{Namikawa}, we prove the following partial result toward the conjecture \ref{homology:generate:padic:unit:general}:

\begin{thm}\label{homology:generate:padic:unit}
Assume $(\mathbf{Non\text{-}Eis})$, 
then $\mathcal{L}_f(\phi)$ is a $p$-adic unit for some Hecke characters $\phi\in\mathfrak{Y}$.
\end{thm}
\begin{proof}
Let us recall that $\operatorname{pr}^1:\ovl{\Gamma^1_1(\fN)}\rightarrow H_1(X_1^1(\fN),\Z)$ is a surjective homomorphism defined by $\gamma\mapsto\{0,\gamma\cdot 0\}_{\ovl{\Gamma^1_1(\fN)}}$, which is trivial for the parabolic elements. First, let us assume that $\Lambda^{(1)}(\phi)\cap\eta_{f,c,1}=0$ in $\ovl{\F}_p$ for any $\phi\in\mathfrak{Y}$. Choose $\gamma=\begin{spmatrix} a & b \\ c & d \end{spmatrix}\in\Gamma_{1}^1(\mathfrak{M})$. Then we can modify $\gamma$ by using the parabolic element $\begin{spmatrix} 1 & t-b \\ 0 & 1 \end{spmatrix}$ of $\Gamma_{1}^1(\mathfrak{M})$:
$$
\begin{pmatrix} a & b \\ c & d \end{pmatrix}\begin{pmatrix} 1 & t-b \\ 0 & 1 \end{pmatrix}=\begin{pmatrix} a & b+a(t-b) \\ c & d+c(t-b) \end{pmatrix}\in\Gamma_{1}^1(\mathfrak{M}),
$$
where $t$ is an element of $\O_F$ coprime to $\mathfrak{M}_p$. Then $b+a(t-b)\in t+\fM_p$, thus we can assume that $b$ is coprime to $\fM_p$. Let us choose $\pi\in P_{b,d}$. Let $C_\pi$ be the set of finite order Hecke characters whose conductor dividing $\pi\O_F$. Then $|C_\pi|$ divides $h_F(N(\pi)-1)$, thus $|C_\pi|$ is a $p$-adic unit. By the orthogonality of characters, we have
\begin{align*}
\operatorname{pr}^1(\gamma)=\Big\{0,\frac{b}{\pi}\Big\}_{\ovl{\Gamma^1_1(\fN)}}=\frac{1}{|C_{\pi}|}\sum_{\mathbf{1}\neq\phi\in C_{\pi}}\ovl{\phi}(b_{\pi\O_F})\Lambda^{(1)}(\phi)+\frac{1}{|C_{\pi}|}\sum_{u\in R_{\pi\O_F}}\xi_{u}^{(1)}
\end{align*} 
in $H_1(X_1^1(\fN),C_1^1(\fN),\O[\phi])$. Note that $C_\pi\subset\mathfrak{Y}$ by the definition. Therefore, by the equation (\ref{pairing:equality}), Lemma \ref{periods:compare} and our assumption, we have 
$$
\operatorname{pr}^1(\gamma)\cap\eta_{f,1}=\Big\{0,\frac{b}{\pi}\Big\}_{\ovl{\Gamma^1_1(\fN)}}\cap\frac{\Omega_{f,\operatorname{c}}}{\Omega_{f}}\eta_{f,{\operatorname{c}},1}=\frac{1}{|C_{\pi}|}\sum_{u\in R_{\pi\O_F}}\xi_{u}^{(1)}\cap\frac{\Omega_{f,\operatorname{c}}}{\Omega_{f}}\eta_{f,{\operatorname{c}},1}
$$
in $\ovl{\F}_p$, which is independent on $b$. Note that $\begin{spmatrix} 1 & 1 \\ \pi-1 & \pi \end{spmatrix}\in\Gamma_1^1(\fN)$, thus we have
$$
\operatorname{pr}^1(\gamma)\cap\eta_{f,1}=\Big\{0,\frac{1}{\pi}\Big\}_{\ovl{\Gamma^1_1(\fN)}}\cap\frac{\Omega_{f,\operatorname{c}}}{\Omega_{f}}\eta_{f,{\operatorname{c}},1}=\operatorname{pr}^1\Big(\begin{spmatrix} 1 & 0 \\ \pi-1 & 1 \end{spmatrix}\begin{spmatrix} 1 & 1 \\ 0 & 1 \end{spmatrix}\Big)\cap\frac{\Omega_{f,\operatorname{c}}}{\Omega_{f}}\eta_{f,{\operatorname{c}},1}=0
$$
in $\ovl{\F}_p$ since $\begin{spmatrix} 1 & 0 \\ \pi-1 & 1 \end{spmatrix}$ and $\begin{spmatrix} 1 & 1 \\ 0 & 1 \end{spmatrix}$ are parabolic elements of $\Gamma_1^1(\fN)$. By the property of the map pr$^1$ and Namikawa \cite[Corollary A.2]{Namikawa}, we have $c_1\cap\eta_{f,1}=0$ in $\ovl{\F}_p$ for each $c_1\in H_1(X_1^1(\fN),\Z)$. Then from the extension of coefficients $H_1(X_1(\fN),\Z)\otimes\ovl{\F}_p\cong H_1(X_1(\fN),\ovl{\F}_p)$ and the Poincar\'{e} duality, we conclude that $\eta_{f,1}$ is trivial, which is a contradiction since $f$ is an eigenform. Therefore, $\Lambda^{(1)}(\phi)\cap\eta_{f,{\operatorname{c}},1}$ is a $p$-adic unit for some $\phi\in\mathfrak{Y}$. 

Let us assume the contrary of Theorem \ref{homology:generate:padic:unit}, i.e., assume that $\mathcal{L}_f(\phi)=0$ in $\ovl{\F}_p$ for any $\phi\in\mathfrak{Y}$. Note that the conductor of $\psi\phi$ and $\phi$ are same for each $\psi\in H_F$, thus we have $\psi\phi\in\mathfrak{Y}$ for $\psi\in H_F$ and $\phi\in\mathfrak{Y}$. Therefore, we have
\begin{align*}
0&=\frac{1}{h_F}\sum_{\psi\in H_F}\mathcal{L}_f(\psi\phi)=\sum_{i=1}^{h_F}\bigg(\frac{1}{h_F}\sum_{\psi\in H_F}\psi(a_i)\bigg)\sum_{u\in R_{\mathfrak{f}(\phi)}} \phi(a_i)\phi(\varpi_{\mathfrak{f}(\phi)} u)\xi_u^{(i)}\cap\eta_{f,{\operatorname{c}},i} \\
&=\sum_{u\in R_{\mathfrak{f}(\phi)}} \phi(a_1)\phi(\varpi_{\mathfrak{f}(\phi)} u)\xi_u^{(1)}\cap\eta_{f,{\operatorname{c}},1}=\Lambda^{(1)}(\phi)\cap\eta_{f,{\operatorname{c}},1}
\end{align*}
in $\ovl{\F}_p$ for any $\phi\in\mathfrak{Y}$,
which is a contradiction.
\end{proof}

\section{$p$-adic $L$-functions and $\mu$-invariants}\label{p:adic:l:function:mu:invariants}
In this section, we define the Bianchi modular version of Mazur-Tate-Teitelbaum $p$-adic $L$-function (see \cite{mazur1986padic}). Also we prove the partial result toward the Bianchi modular version of Greenberg's conjecture. 

\subsection{$p$-adic integral}
Let $m>0$ be an integer. Let us fix a normalized eigenform $f\in S_2(\fN,\chi)$. Let $p$ be a odd prime number coprime to the following ideal 
$$
h_FD_F|\O_F^\times|\mathfrak{h}_n\fN\prod_{i=1}^{h_F}\fa_i^{-1},
$$
where $\mathfrak{h}_n$ is the integer defined in the Remark \ref{horizontal:module:indices}.
Let $\fp$ be an ordinary 
prime ideal lying above $p$. Let us denote $W_\fp$ the torsion subgroup of $\O_{F,\fp}^\times$. Let $\a_{f,\fp}$ be the $p$-adic unit root of the equation $x^2-a_f(\fp)x+\chi(\varpi_\fp)N(\fp)=0$. Let us set $\phi$ by a finite order Hecke character of $\fp$-power conductor. Let $\O:=\Z_p[\a_{f,\fp},\mu_{h_F},\phi,\chi]\Z_p[\{a_f(\fl):\fl\}]$. Let $\varpi$ be a uniformizer of $\O$. 
Define a function $\nu_f$ on the set of the basic open subsets of $\O_{F,\fp}$ as follows:
$$
\nu_f(a+\fp^m\O_{F,\fp}):=\a_{f,\fp}^{-m}\sum_{i=1}^{h_F}\big(\xi^{(i)}_{m}(a)-\chi(\varpi_\fp)\a_\fp^{-1}\xi^{(i)}_{m-1}(a)\big)\cap\eta_{f,{\operatorname{c}},i}\in\O.
$$
Then we have following:
\begin{prop}\label{padic:measure}
$\nu_f$ is a $p$-adic measure on $\O_{F,\fp}$.
\end{prop}
\begin{proof}
Note that there are $h_F$-permutations $\sigma$ and $\tau$ such that 
$$
\a_{i}\fa_i\fp^{-1}=\fa_{\sigma(i)},\ \beta_{i}\fa_i\fp=\fa_{\tau(i)}.
$$
Note that $\tau=\sigma^{-1}$ and $\beta_{\sigma(i)}=\a_{i}^{-1}$.
Let us denote by $r(\fp)$ the matrix $\begin{spmatrix} 1 & 0 \\ 0 & \varpi_\fp \end{spmatrix}$. For $\a\in F^\times$, let us denote $\gamma_\a:=\begin{spmatrix} \a & 0 \\ 0 & 1 \end{spmatrix}$ and $w_\a:=\a^{(\infty)}\varpi_{\a\O_F}^{-1}\in\widehat{\O}_F^\times$. As $(a_i)_\fp=1$ for each $i$, we have $(w_{\a_i})_\fp=(w_{\beta_i})_\fp=1$. Then by Proposition \ref{adelic:modular:symbol:reprn}, we obtain that
\begin{align*}
&\gamma_{\a_i}t_i(0;y)t\Big(\frac{a}{\varpi_\fp^m}\Big)r(\fp)=\a_{i,\infty}^{\frac{1}{2}}\varpi_\fp t_{\sigma(i)}(0;|\a_{i}|y)\Big(0;\frac{\a_{i}}{|\a_{i}|}\Big)t\Big(\frac{a}{\varpi_\fp^{m-1}}\Big)\begin{pmatrix} w_{\a_i} & 0 \\ 0 & 1 \end{pmatrix}, \\
&\gamma_{\beta_i}t_i(0;y)t\Big(\frac{a}{\varpi_\fp^m}\Big)\begin{pmatrix} \varpi_\fp & c \\ 0 & 1 \end{pmatrix}=\beta_{i,\infty}^{\frac{1}{2}}t_{\tau(i)}(0;|\beta_{i}|y)\Big(0;\frac{\beta_{i}}{|\beta_{i}|}\Big)t\Big(\frac{a+\varpi_\fp^m c}{\varpi_\fp^{m+1}}\Big)\begin{pmatrix} w_{\beta_i} & 0 \\ 0 & 1 \end{pmatrix}
\end{align*}
in $\GL_2(\A_F)$ for each $c\in\O_{F,\fp}$. Note that $\begin{spmatrix} w_{\a_i} & 0 \\ 0 & 1 \end{spmatrix}$, $\begin{spmatrix} w_{\beta_i} & 0 \\ 0 & 1 \end{spmatrix}\in U_1(\fN)$.
From the definition of the Hecke operator $T_\fp$ (Hida \cite[Section 4]{hida1994critical}), $T_\fp(f)$ can be written as
$$
T_\fp(f)(g):=f\big(gr(\fp)\big)+\sum_{c\text{ mod }{\fp}}f\bigg(g\begin{pmatrix} \varpi_\fp & c \\ 0 & 1 \end{pmatrix}\bigg).
$$
Together with the equation (\ref{integral:evaluation}) and the above equations, we obtain
\begin{align*}
a_f(\fp)\xi_{m}^{(i)}(a)\cap\iota\big(\delta^{(i)}(f_{(i)})\big)=&\chi(\varpi_\fp)\xi_{m-1}^{(\sigma(i))}(a)\cap\iota\big(\delta^{(\sigma(i))}(f_{(\sigma(i))})\big) \\
&+\sum_{c\text{ mod }{\fp}}\xi_{m+1}^{(\tau(i))}(a+\varpi_\fp^m c)\cap\iota\big(\delta^{(\tau(i))}(f_{(\tau(i))})\big).
\end{align*}
From the above equation, we obtain
$$
\sum_{b\equiv a\text{ mod }\fp^m}\nu_f(b+\fp^{m+1}\O_{F,\fp})=\nu_f(a+\fp^{m}\O_{F,\fp}),
$$
which implies that $\nu_f$ is a $p$-adic distribution on $\O_{F,\fp}$.
\end{proof}

Let $\langle \cdot\rangle:\O_{F,\fp}^\times\rightarrow 1+\fp\O_{F,\fp}$ be the projection map. Let us define the $p$-\emph{adic L-function} $L_\fp(s,f,\phi)$ of $f$ by the Mazur-Mellin transformation of $\nu_f$:
$$
L_\fp(s,f,\phi):=\int_{\O_{F,\fp}^\times}\langle x\rangle^{-s}\phi(x)d\nu_f(x) .
$$
Then our $p$-adic $L$-function interpolates the integral $L$-values $\mathcal{L}_f(\phi)$: 
\begin{prop}\label{integral:lvalue:interpolation}
If the Hilbert character part of $\phi$ is trivial and the conductor of $\phi$ is $\fp^m$ for $m\geq 1$, then we have
$$
L_\fp(0,f,\phi)=\a_{\fp}^{-m}\mathcal{L}_f(\phi).
$$
\end{prop}
\begin{proof}
From our assumption, the value $L_\fp(0,f,\phi)$ is given by
$$
\a_{f,\fp}^{-m}\mathcal{L}_f(\phi)-\a_{f,\fp}^{-m-1}\chi(\varpi_\fp)\sum_{i=1}^{h_F}\sum_{a\in (\O_F/\fp^{m})^\times}\phi(a)\xi_{m-1}^{(i)}(a)\cap\eta_{f,\operatorname{c},i}.
$$
Then the latter term of the above equation can be written as follows:
$$
-\a_{f,\fp}^{-m-1}\chi(\varpi_\fp)\sum_{i=1}^{h_F}\sum_{b\in (\O_F/\fp^{m-1})^\times}\sum_{d\text{ mod }\fp}\phi(b)\phi(1+b^{-1}d\varpi_\fp^{m-1})\xi_{m-1,\fN}^{(i)}(b)\cap\eta_{f,\operatorname{c},i}.
$$
By the orthogonality of characters, the above equation is equal to zero. Therefore, we obtain the desired result.
\end{proof}
Let $\g$ be a topological generator of $1+\fp\O_{F,\fp}$, then we have the isomorphism $(y\mapsto\g^y):\O_{F,\fp}\cong 1+\fp\O_{F,\fp}$. By the decomposition $\O_{F,\fp}^\times=W_\fp\times(1+\fp\O_{F,\fp})$, the $p$-adic Iwasawa power series $L_\fp(T;f,\phi)$ of $L_\fp(s,f,\phi)$ is defined by
$$
L_\fp(T;f,\phi):=\int_{\O_{F,\fp}}T^y\phi(\g^y)\sum_{\k\in W_\fp}\phi(\k)d\nu_f(\k\g^y)\in\O[[T-1]].
$$
Then we have $L_\fp(\gamma^{-s};f,\phi)=L_\fp(s,f,\phi)$. By the Weierstrass preparation theorem, there exists a product $P(T)\in\O[[T-1]]$ of distinguished polynomials, a unit element $U(T)\in\O[[T-1]]^\times$ and an integer $\mu$ such that $L_\fp(T;f,\phi)=\pi^{\mu}P(T)U(T)$. Let us define the $\mu$-\emph{invariant} $\mu\big(L_\fp(s,f,\phi)\big)$ of $L_\fp(s,f,\phi)$ by $\mu$.

\subsection{Results toward the vanishing of $\mu$-invariant} 
By using the results in the previous subsection, we can suggest the Bianchi modular version of Greenberg conjecture :
\begin{conj}\label{greenberg:conj:bianchi:version} Assume $(\mathbf{Non\text{-}Eis})$, then $\mu\big(L_\fp(s,f,\phi)\big)=0$.
\end{conj}

In the main theorem, we discuss the $\mu$-invariant of $L_\fp(s,f\otimes\psi,\phi)$ for $\psi\in H_F$ instead of $\mu\big(L_\fp(s,f,\phi)\big)$ due to the technical issues. To achieve this, we need to compare the periods of $f$ and $f\otimes\psi$:

\begin{prop}\label{compare:periods} Let $\a_{f,\psi,\operatorname{c}}\in\C_p$ be the ratio of the periods:
$$
\a_{f,\psi,\operatorname{c}}:=\frac{\Omega_{f\otimes\psi,\operatorname{c}}}{\Omega_{f,\operatorname{c}}}
$$
Then $\a_{f,\psi,\operatorname{c}}\in\O^\times$ for any $\psi\in H_F$.
\end{prop}

\begin{proof} Let us choose $\psi\in H_F$. By the definition, $\eta_{f\otimes\psi,\operatorname{c}}$ is a $\O$-generator of $H^1_{\text{cusp}}(X_1(\fN),\O)[f\otimes\psi]$. Note that $f\otimes\psi$ is a normalized Hecke eigenform.
Let us define the cohomology classes $\eta_{f,\psi,\operatorname{c}}$ by
$$
\eta_{f,\psi,\operatorname{c}}:=\sum_{i=1}^{h_F}\psi(a_i)\eta_{f,\operatorname{c},i}.
$$
We observe that the class $\eta_{f,\psi,\operatorname{c}}$ is an element in $H_{\text{cusp}}^1(X_1(\fN),\O)$. By Lemma \ref{cohomology:hecke:action}, we obtain $T_\fq\eta_{f,\psi,\operatorname{c}}=a_f(\fq)\psi(\varpi_\fq)\eta_{f,\psi,\operatorname{c}}$ for almost all prime ideals $\fq$. Thus by the strong multiplicity one theorem, we have 
$$
\eta_{f,\psi,\operatorname{c}}\in H_{\text{cusp}}^1(X_1(\fN),\O)[f\otimes\psi],
$$ 
which implies that $\eta_{f,\psi,\operatorname{c}}=\beta_{f,\psi,\operatorname{c}}\eta_{f\otimes\psi,\operatorname{c}}$ for some $\beta_{f,\psi,\operatorname{c}}\in\O$. 

On the other hand, by the definition of $f\otimes\psi$, we have
$$
\a_{f,\psi,\operatorname{c}}\eta_{f\otimes\psi,\operatorname{c}}=\frac{\iota\big(\delta(f\otimes\psi)\big)}{\Omega_{f,\operatorname{c}}}=\sum_{i=1}^{h_F}\psi(a_i)\frac{\iota\big(\delta^{(i)}(f_{(i)})\big)}{\Omega_{f,\operatorname{c}}}=\eta_{f,\psi,\operatorname{c}},
$$
which implies that $\a_{f,\psi,\operatorname{c}}=\beta_{f,\psi,\operatorname{c}}\in\O.$
Since $(f\otimes\psi)\otimes\psi^{-1}=f$, we have 
$\a_{f,\psi,\operatorname{c}}^{-1}=\a_{f\otimes\psi,\psi^{-1},\operatorname{c}}\in\O$. 
Therefore, we conclude that $\a_{f,\psi,\operatorname{c}}$ is a $p$-adic unit.
\end{proof}

Let $n\geq\max\{22,m\}$ be an integer, where $m$ is the integer such that $\fp^m$ is the conductor of $\phi$. Let $\mathfrak{X}_n$ be the arithmetic progression defined in Theorem \ref{full:rank:horizontal}. 
Let $\omega:W_\fp\rightarrow\O_{F,\fp}^\times$ be the Teichm{\"u}ller character. Note that we can consider $\omega$ to be a Hecke character of conductor $\fp$.
Using the results on the full-rankness and the periods, we obtain the following result toward Conjecture \ref{greenberg:conj:bianchi:version}:

\begin{thm}\label{mu:invariant:vanishing}  
Assume $(\mathbf{Non}$-$\mathbf{Eis})$. If $\fp\in\mathfrak{X}_n$, then we have $\mu\big(L_\fp(s,f\otimes\psi,\phi\omega^j)\big)=0$ for some $\psi\in H_F$ and $0\leq j<p-1$.
\end{thm}

\begin{proof} We observe that $f\otimes\psi$ is a $\fp$-ordinary normalized eigenform for any $\psi\in H_F$. Assume the contrary, in other words, $\mu\big(L_\fp(s,f\otimes\psi,\phi\omega^j )\big)>0$ for any $\psi\in H_F$ and integers $0 \leq j<p-1$. Then for each sufficiently small open subset $U$ of $\O_{F,\fp}$, we have 
\begin{equation}\label{padic:integration:equation}
\int_{U}\phi\omega^j(\g^y)\sum_{\k\in W_\fp}\phi\omega^j(\k)d\nu_{f\otimes\psi}(\k\g^y)\equiv 0\ (\text{mod }\varpi).
\end{equation}
Thus for each $a\equiv 1\ (\text{mod }\fp)$, we obtain
\begin{equation}\label{full:generation:modp:equation}
\sum_{\k\in W_\fp}\phi\omega^j(\k)\nu_{f\otimes\psi}(a\k+\fp^{n}\O_{F,\fp})\equiv 0\ (\text{mod }\varpi).
\end{equation}
Note that we have $\a_{f\otimes\psi,\fp}=\a_{f,\fp}$ for each $\psi\in H_F$ as $\psi(\varpi_\fp)=1$. By using the congruence $a_{f\otimes\psi}(\fp)\equiv\a_{f,\fp}\ (\text{mod }\varpi)$, we can write the equation (\ref{full:generation:modp:equation}) by
\begin{equation}\label{full:generation:modp:equation:2}
\sum_{i=1}^{h_F}\sum_{\k\in W_\fp}\phi\omega^j(\kappa)\big(a_f(\fp)\xi^{(i)}_{n}(a\kappa)-\chi(\varpi_\fp)\xi^{(i)}_{n-1}(a\kappa)\big)\cap\eta_{f\otimes\psi,{\operatorname{c}},i}\equiv 0\ (\text{mod }\varpi).
\end{equation}
From Proposition \ref{compare:periods}, we can replace the period $\Omega_{f\otimes\psi,c}$ by $\Omega_{f,c}$, thus we can replace $\eta_{f\otimes\psi,c,i}$ by $\psi(a_i)\eta_{f,c,i}$. By applying the following operator
$$
\frac{1}{h_F(p-1)}\sum_{\psi\in H_F}\sum_{j^\p=0}^{p-2}\omega^{-j^\p}(\nu)
$$ 
for $\nu\in W_\fp$ on the equation (\ref{full:generation:modp:equation:2}), we obtain
\begin{equation}\label{full:generation:modp:equation:3}
\big(a_f(\fp)\xi^{(1)}_{n}(a\nu)-\chi(\varpi_\fp)\xi^{(1)}_{n-1}(a\nu)\big)\cap\eta_{f,\operatorname{c},1}\equiv 0\ (\text{mod }\varpi)
\end{equation}
for each $a\equiv 1\ (\text{mod }\fp)$ and $\nu\in W_\fp$.
Hence for each $c\in H_1(X_1^1(\fN),\ovl{\F}_p)$, we have
$$
c\cap\eta_{f,1}=\sum_{a,\nu}c_{a,\nu}\big(a_f(\fp)\xi^{(1)}_{n}(a\nu)-\chi(\varpi_\fp)\xi^{(1)}_{n-1}(a\nu)\big)\cap\frac{\Omega_{f,\operatorname{c}}}{\Omega_{f}}\eta_{f,\operatorname{c},1}\equiv 0\ (\text{mod }\varpi)
$$
due to the equations (\ref{pairing:equality}), (\ref{full:generation:modp:equation:3}), Proposition \ref{full:rank:horizontal:variation} and Lemma \ref{periods:compare}. Note that the pairing $\cap$ which is defined in the equation (\ref{lefschetz:poincare:pairing}) is compatible for $R=\O$ and $R=\ovl{\F}_p$ under the projection map $(\text{mod }\pi):\O\rightarrow\ovl{\F}_p$ by the assumption on $\fp$. Therefore, it is a contradiction since $f$ is normalized eigenform, $\eta_f$ is chosen to be $p$-optimal, and the pairing $\cap$ is non-degenerate.
\end{proof}

\begin{rem}
Note that if there is no ordinary prime ideals $\fp$ in $\mathfrak{X}_n$, then our main theorem is vacuously true.
Therefore, we should verify that the set of ordinary prime ideals has positive density.
\end{rem}



\section{Density of ordinary primes} 
Let $\ell$ be a rational prime. By using the result of Serre \cite{serre1981quelques} and the prime number theorem, the set of ordinary rational primes has density one for any classical elliptic modular forms with no CM. In this section, we prove the similar result.


Let $f\in S_2\big(U_0(\fN)\big)=S_2(\fN,\mathbf{1})$ be a normalized Hecke eigenform, which is neither CM nor the lifting of classical elliptic modular newform. For simplicity, we call such $f$ \emph{Genuine newform}. Let us denote by $\Q_f$ the coefficient field of $f$  By Friedberg-Hoffstein \cite[Theorem A, Theorem B]{friedberg1995nonvanishing} and Taylor \cite[Theorem 2, Corollary 4]{taylor1994ladic}, we obtain a continuous irreducible representation 
$$
\rho_{f,\ell}:G_F\rightarrow\GL_2(\Q_\ell)
$$
such that if $v$ is an unramified place of $F$ which does not divide $\fN\ell$, then 
\begin{enumerate}
\item $\rho_{f,\ell}$ is unramified at $v$,
\item Im$(\rho_{f,\ell}\oplus\rho^{\co}_{f,\ell})$ is dense in $\{(a,b)\in\GL_2^2:\det(a)=\det(b)\}$, and
\item either $$\text{char}\big(\rho_{f,\ell}(\text{Frob}_v)\big)(X)=X^2-a_f(v)X+N(v)$$ or $a_f(v^{\co})=0$ and 
$$\text{char}\big(\rho_{f,\ell}(\text{Frob}_v)\big)(X)=X^2+N(v),$$ 
where the first possibility occurs outside a set $\{v\}_v$ of density zero.
\end{enumerate}
From this, we have the following proposition:
\begin{prop}\label{density:nonzero:trace:places}
For $x>0$, the number of places $v$ of $F$ such that $\operatorname{Tr}\big(\rho_{f,\ell}(\operatorname{Frob}_v)\big)=0$ and $N(v)<x$ is given by 
$$
O\bigg(\frac{x}{\operatorname{log}(x)^{\frac{3}{2}-\e}}\bigg)
$$ 
for any $\e>0$.
\end{prop}
\begin{proof}
Let $G:=\text{Im}(\rho_{f,\ell})/\big(\Q_\ell^\times\cap\text{Im}(\rho_{f,\ell})\big)$, $C$ the image of the set $\{s\in\text{Im}(\rho_{f,\ell}):\text{Tr}(s)=0\}$ in $G$. Note that Im$(\rho_{f,\ell})$ is dense in $\GL_2$, thus $\dim_{\Q_\ell}(G)=3$ and $\dim_{\Q_\ell}(C)\leq 2$. By the argument of Serre \cite[Th\'{e}or\`{e}me 15]{serre1981quelques}, we have 
$$
\inf_{s\in C}\dim_{\Q_\ell}(\text{Ad}(s)\rm{Id})=2,
$$ 
where $\text{Ad}$ is the adjoint representation of the Lie algebra $\mathfrak{g}$ of $G$ induced from $\rho_{f,\ell}$. Applying \cite[Th\'{e}or\`{e}me 12]{serre1981quelques} for a pair $(G,C)$, we obtain the desired estimation. 
\end{proof}
From this, we obtain the result toward Conjecture \ref{greenberg:conj:bianchi:version}:
\begin{thm}\label{density:result}
 If $\Q_f=\Q$, then $f$ is $\fl$-ordinary for a set $\{\fl\}_\fl$ of prime ideals of $F$ of density one.
\end{thm}
\begin{proof}
Let $\fl$ be a non-ordinary prime ideal of $F$ lying above a rational prime $\ell$ such that $N(\fl)>2^\frac{64}{25}$. Then we obtain $a_f(\fl)\in N(\fl)\Z$. Assume that $a_f(\fl)\neq 0$, then we have 
$$
N(\fl)\leq |a_f(\fl)|\leq 2N(\fl)^{\frac{39}{64}}
$$
by Nakasuji \cite[Corollary 1.1]{Nakasuji}, which is a contradiction. Therefore, we have $a_f(\fl)=0$ if $\fl$ is non-ordinary and $N(\fl)>2^\frac{64}{25}$. The converse direction is clear. Therefore, if $N(\fl)$ is sufficiently large, then $a_f(\fl)$ is zero if and only if $\fl$ is non-ordinary.

On the other hand, by the prime number theorem and Proposition \ref{density:nonzero:trace:places}, the density of prime ideals $\fl$ of $F$ such that $\text{Tr}\big(\rho_{f,\ell}(\text{Frob}_\fl)\big)=0$ is given by
$$
\lim_{x\rightarrow\infty} \operatorname{log}(x)^{\e-\frac{1}{2}}=0.
$$
Note that the equation $\text{Tr}\big(\rho_{f,\ell}(\text{Frob}_\fl)\big)=a_f(\fl)$ holds for a set of prime ideals $\{\fl\}_\fl$ of density one. Hence we obtain the desired result.
\end{proof}

From Theorem \ref{mu:invariant:vanishing} and Theorem \ref{density:result}, we obtain the main result of the present paper:

\begin{thm}\label{main:result} Let $f\in S_2\big(U_0(\fN)\big)$ be a Genuine newform such that $\Q_f=\Q$. Let $\phi$ be a finite order Hecke character of $\fp$-power conductor. Assume $(\mathbf{Non}$-$\mathbf{Eis})$, then for a positive proportion of ordinary primes $\fp$, we have $\mu\big(L_\fp(s,f\otimes\psi,\phi\omega^j)\big)=0$ for some $\psi\in H_F$ and $0\leq j<p-1$. 
\end{thm}

\section*{Acknowledgements}
\thispagestyle{empty}

This work was supported by the National Research Foundation of Korea (NRF) grant funded by the Korea government(MSIT)(No. 2019R1A2C108860913).
The author would like to thank his advisor Hae-Sang Sun for helpful discussions, comments, and suggesting the research topic. 
The author also would like to thank Ashay Burungale, Ming-Lun Hsieh, Chan-Ho Kim, Kenichi Namikawa and Chol Park and for helpful discussions and comments. 
The author plans to generalize the Hecke field generation results by Kim-Sun \cite{kim2017modular} and Sun \cite{sun2019generation} to the case of Bianchi modular forms.

\end{document}